\documentclass[12pt]{article}
\usepackage[english]{babel}
\usepackage{comment}
\usepackage[utf8]{inputenc}
\usepackage[T1]{fontenc}
\usepackage{eufrak, color, mathrsfs,dsfont, bbm}
\usepackage{amsmath, amssymb, amsthm, pdfsync}
\usepackage{mathtools, mathabx, cases, braket, upgreek}
\usepackage{xcolor}
\usepackage{babelbib}
\usepackage{graphicx} 
\usepackage[left=2cm, right=2cm, top=2.5cm, bottom=2.5cm]{geometry}
\usepackage{math}
\mathtoolsset{showonlyrefs}

\usepackage{xcolor}
\usepackage[colorlinks=true, linkcolor=purple, citecolor=purple]{hyperref}

\numberwithin{equation}{section}

\theoremstyle{plain}
\newtheorem{theorem}{Theorem}[section]
\newtheorem{lemma}[theorem]{Lemma}
\newtheorem{proposition}[theorem]{Proposition}
\newtheorem{corollary}[theorem]{Corollary}
\theoremstyle{definition}
\newtheorem{definition}[theorem]{Definition}

\theoremstyle{remark}
\newtheorem{remark}[theorem]{Remark}

\renewcommand{\epsilon}{\varepsilon}
\newcommand{\Opw}{\operatorname{Op^w}}
\newcommand{\Oph}{\operatorname{Op}_h^{\operatorname{w}}}

\renewcommand{\div}{\textup{div}}

\newcommand{\OpM}{\operatorname{Op^M}}
\newcommand{\Addresses}{
  \bigskip

 Gabriel Rivi\`ere\par\nopagebreak \textsc{Laboratoire de Math\'ematiques Jean Leray (UMR CNRS 6629), Nantes Universit\'e, 2 rue de la Houssini\`ere, 44322 Nantes Cedex 03, France}\par\nopagebreak
  \par\nopagebreak
  \textit{E-mail address:} \texttt{gabriel.riviere@univ-nantes.fr} 

  \bigskip

  Maria Teresa Rotolo \par\nopagebreak\textsc{Scuola Internazionale Superiore di Studi Avanzati (SISSA), via Bonomea 265, 34136 Trieste, Italy}\par\nopagebreak
  \textit{E-mail address:} \texttt{mrotolo@sissa.it}

}

\title{Sobolev instability for perturbations of periodic transport equations}

\author{Gabriel Rivi\`ere and Maria Teresa Rotolo}

\begin{document}

\maketitle
\begin{abstract}
We consider linear and time-dependent perturbations of periodic transport equations on the two-dimensional torus. For generic perturbations, we prove the existence of a large class of initial data whose Sobolev norms diverge exponentially fast. In higher dimensions, this remains true under a Morse-Smale assumption on the resonant part of the perturbation. In both cases, this is achieved by a normal form procedure and by studying Sobolev instabilities for time-dependent perturbations of Morse-Smale transport equations. The latter are analyzed on general compact manifolds using techniques from microlocal analysis and hyperbolic dynamics.
\end{abstract}


\section{Introduction and main results}

In this article we construct unstable solutions for certain classes of linear and time dependent transport equations, through the analysis of their Sobolev norms. The study of instabilities is, in its full generality, a standard topic in linear and nonlinear partial differential equations, which has many facets. In particular, a widely used method is to look at the evolution in time of the Sobolev norms~\eqref{sobolev.spaces} of solutions: an unbounded growth in time of the positive Sobolev norms is an indicator of energy transfer phenomena, in the form of energy cascades towards higher frequencies. A class of equations that we will study consists of small-in-size, time dependent perturbations of Morse-Smale transport equations on compact manifolds (see equation \eqref{transport.eq}). However, since our analysis is motivated by applications to the dynamics of integrable partial differential equations and recent results in KAM theory, we begin by giving an application of our main theorem in this direction. 

Let $n\geq 1$ and denote by $\mathbb{T}^{n+1}$ (with $\T:=\mathbb{R}/2\pi \Z$) the Euclidean $(n+1)$-dimensional torus. Given $\nu\in\R^n$ and $\omega\in\R$, we consider the following equation
\begin{equation}\label{transport.torus}
 \partial_tu_\e= \left(\nu+ \e V (\omega t,x)\right)\cdot\nabla_xu_\e +\frac{\e}{2}\text{div}_x(V(\omega t,x))u_\e,\quad x\in\T^n,\ t\in\R,\ \e>0,   
\end{equation}
where $V (t,x)\in C^{\infty}(\T \times \T^n; \R^n)$ is a time-dependent vector field and where $\text{div}(\cdot)$ denotes the divergence with respect to the standard Euclidean volume on $\T^n$. In particular, note immediately that the transport operator on the right-hand side of the equation is formally skew adjoint on $L^2(\T^n)$ for every $t\in\R$. Thus one has 
$$
\|u_\e(t)\|_{L^2}=\|u_\e(0)\|_{L^2}, \quad \forall t\in\R,\ \forall\e\geq 0.
$$
Moreover, it is worth noting that, for $\e=0$, all the Sobolev norms are preserved, i.e.
$$
 \|u_0(t)\|_{\sigma}=\|u_0(0)\|_{\sigma}, \quad \forall t\in\R,\ \forall \sigma \in \R,\
$$
where $\|.\|_\sigma$ denotes the standard Sobolev norm on $\T^n$ (see~\eqref{sobolev.spaces} below). It is natural to ask if this property is stable under perturbation, i.e. for $\e>0$ small enough. In~\cite[Th.~2.4]{BambusiLangellaMontalto} (see also~\cite{FeolaGiulianiMontaltoProcesi}), it is proved that, for generic choices of parameters $(\omega,\nu)$ in $[1,2]^{n+1}$ satisfying some Diophantine conditions,  equation~\eqref{transport.torus} is \emph{reducible}. Roughly speaking, it means that, up to a bounded transformation on the Sobolev space $\mathcal{H}^\sigma(\T^n)$, one can conjugate the equation to a diagonal one in the Fourier basis (up to smoothing remainders). Such choices of the parameters $(\omega, \nu)$ fall in what is usually called \emph{non-resonant} regime, and a direct consequence of this result is that either there exist solutions whose positive Sobolev norms blow up exponentially (in the case where the diagnonal operator has eigenvalues with nonzero real parts) or the Sobolev norms $\|u_\e(t)\|_\sigma$ remain bounded (for instance if $V$ is even). See~\cite[Cor.~2.5]{BambusiLangellaMontalto} for details.

In the present work, we consider instead a \emph{resonant} regime, i.e., we choose $\omega=1$ and $\nu \in \N^n$, so that the potential $V$ has the same time-period as the solution to the unperturbed equation which may create resonance phenomena. 
Using methods coming from the theory of hyperbolic dynamical systems, we prove that, in the case $n=2$ and for generic potentials $V$, there always exist solutions whose positive Sobolev norms grow exponentially. 
\begin{theorem}\label{theorem.torus} 
Let $n=2$, $\omega = 1$ and $\nu\in\N^2\setminus\{0\}$. Then there exists an open and dense subset $\mathcal{O}_\nu$ of~\footnote{The set is endowed with its standard Fr\'echet topology.} $C^{\infty}(\T^3;\R^2)$ such that, for every $V\in \mathcal{O}_\nu$, the following holds.

There exists $\delta_0>0$ such that, for every $0 \leq \sigma \leq \delta_0$ and for every $0 \leq \e \leq \delta_0 \sigma^2$, one can find $v_\e \in \cH^\sigma(\T^2)$ and $r:=r(\e, \sigma, v_\e)>0$ such that 
\begin{equation}
    \norm{u_\e(0) - v_\e}_\sigma \leq r \quad \implies \quad \norm{u_\e(t)}_\sigma \geq \delta_0 e^{\delta_o \sigma t\e} \norm{u_\e (0)}_{\sigma}, \quad \mbox{ for any } t >0, 
\end{equation}
where $u_\e(t)$ is the solution to equation \eqref{transport.torus} with initial datum $u_\e(0)$.
\end{theorem}

Let us immediately remark that, for every potential $V \in C^\infty(\T\times \T^2; \R)$, the Sobolev norms of all solutions to equation \eqref{transport.torus} grow at most exponentially fast with the $L^2$-norm being preserved -- see  \cite[Th.~1.2]{MasperoRobert}. Thus, among other things, we identify a large class of time dependent transport equations for which this exponential growth is indeed sharp. Yet, we emphasize that the growth from this lower bound is effective for times $t\gg\e^{-1}$.

In order to produce instabilities, we will exploit the resonant part of the perturbation $V(t,x)$. Namely, we will rely on the dynamical properties of the \emph{resonant} vector field 
$$
\langle V\rangle_\nu(x):=\frac{1}{2\pi}\int_0^{2\pi}V(\tau, x-\tau \nu)d\tau,\quad x\in\T^2.
$$
Indeed, after a normal form procedure (see Section~\ref{sec.2D}), instabilities for~\eqref{transport.torus} can be obtained from unstable solutions to time dependent perturbations of transport equations generated by $\langle V\rangle_\nu$. In Theorem~\ref{main.theorem} below, we will prove that, as soon as $\langle V\rangle_\nu$ has the Morse-Smale property~\cite{Smale1960}, such unstable solutions exist. Then, a classical result due to Peixoto on the genericity of Morse-Smale vector fields in dimension $2$~\cite{Peixoto1962} will allow us to deduce the genericity of $V$ in $C^{\infty}(\T^3; \R^2)$. Peixoto's result is however peculiar of dimension $2$ and we are not aware of analogous generic sets of vector fields in higher dimension. Hence, in higher dimensions, we will only be able to produce unstable solutions when $\langle V\rangle_\nu$ belongs to the open set of Morse-Smale vector fields which does not generate a dense subset in $C^{\infty}(\T^{n+1};\R^n)$, when $n \geq 3$. See Theorem~\ref{theorem.torus.n} for details. 
We now move towards the statement of a general result on transport equations generated by Morse-Smale vector fields. This will constitute the main technical result used to prove Theorem~\ref{theorem.torus} and it is the content of Theorem \ref{main.theorem} below.

\subsection{Morse-Smale transport equations on manifolds}
We consider a smooth ($C^\infty$), compact and oriented manifold $M$ of dimension $n \geq 1$, without boundary. We endow the manifold $M$ with a Riemannian metric that we denote by $g$ and we let $\upsilon_g$ be the corresponding Riemannian volume on $M$. We denote by $\fX(M)$ the set of smooth vector fields on $M$ and we consider the following class of transport equations:
\begin{equation}\label{transport.eq}
    \pa_t u_\e = \cL_{V+\e \cP(t)} u_\e  +\frac{1}{2}\text{div}_{\upsilon_g}(V+\e\cP(t))u_\e, \quad x \in M, \ t \in \R, \ \e >0,
\end{equation}
where $V \in \fX(M)$, $\cP(t) \in C^0_b(\R; \fX(M))$ and $\text{div}_{\upsilon_g}$ is the divergence with respect to the Riemannian volume form. Here, $\cL$ denotes the usual lie derivative, i.e., the order one differential operator defined by
$
\cL_Y u := \di u [Y],
$
for every $Y \in \fX(M)$, and $C^0_b(\R; \fX(M))$ denotes the set of time dependent vector fields on $M$, which are continuous and bounded in time.  
Note that, as in equation \eqref{transport.torus}, the addition of the divergence term ensures that the operator on the right-hand side of equation \eqref{transport.eq} is skew adjoint for every time $t \in \R$ so that
$$
\|u_\e(t)\|_{L^2(M,\upsilon_g)}=\|u_\e(0)\|_{L^2(M,\upsilon_g)}, \quad \forall t\in\R, \ \forall \e >0.
$$

Our main goal is to prove existence of unstable solutions for equation \eqref{transport.eq}, where we again identify unstable solutions with the ones that exhibit an unbounded growth in time of their positive Sobolev norms. To this aim, recall that the Sobolev spaces $\cH^\sigma$ are defined as
\begin{equation}\label{sobolev.spaces}
    \cH^\sigma(M):= \left\{ u \in L^2(M; \C) : \norm{ ( 1-\Delta_g)^{\frac{\sigma}{2}} u }_{L^2(M,\upsilon_g)} < +\infty \right\}, \quad \sigma \in \R, 
\end{equation}
where $\Delta_g$ is the Laplace-Beltrami operator on $M$. The corresponding Sobolev norm will be denoted by $\|.\|_\sigma$. We emphasize that, also in this general case, one has only the a priori bound for every $\sigma>0$ \cite[Th.~1.2]{MasperoRobert}
$$
 \|u_\e(t)\|_{\sigma}\leq C_\sigma e^{C_\sigma|t|}\|u_\e(0)\|_{\sigma},\quad \forall t\in\R,\ \text{for some}\ C_\sigma>0.
$$
Our main theorem shows that this exponential growth is in fact sharp for the class of Morse-Smale vector fields.

\begin{theorem}\label{main.theorem}
Let $V\in \fX(M)$ and $\cP \in C^0_b(\R, \fX(M))$. Suppose that $V$ is a Morse-Smale vector field (see Definition \ref{def.MS}) and, if $n=1$, suppose in addition that it has at least one critical point.

Then, there exists $\delta_{0}>0$ such that, for every $0\leq \sigma\leq \delta_0$ and for every $0\leq \e \leq \delta_0\sigma^2$, one can find $v_\e\in\mathcal{H}^\sigma(M)$ and $r:=r(\e,\sigma,v_\e)>0$ such that 
\begin{equation}\label{growth}
\|u_\e(0)-v_\e\|_{\sigma}\leq r\quad \implies\quad   \norm{u_\e(t)}_\sigma \geq \delta_0e^{\delta_0\sigma t}\|u_\e(0)\|_\sigma,\quad \text{for any}\ t>0,
\end{equation} 
where $u_\e(t)$ is the solution to equation \eqref{transport.eq} with initial datum $u_\e(0).$ 
\end{theorem}

The class of Morse-Smale vector fields was introduced by Smale in~\cite{Smale1960} and it forms an open subset of $\fX(M)$ which contains an open and dense subset of all gradient vector fields. Yet note that the corresponding flow may have finitely many closed hyperbolic orbits that are not reduced to a point. We refer to Section~\ref{MS.section} for precise definitions and basic properties. In dimension $2$, the set of Morse-Smale vector fields is also dense in $\fX(M)$ thanks to Peixoto's Theorem~\cite{Peixoto1962}. Moreover, these vector fields are structurally stable, meaning that any small (time independent) perturbation of $V$ generates a flow which is topologically conjugate to the initial flow as shown by Palis and Smale~\cite{PalisSmale}.

The overall strategy in order to prove Theorem \ref{main.theorem} consists into constructing a Lyapunov (or escape) function for the lifted symplectic dynamical systems to the cotangent bundle $T^*M$ and into combining this construction with microlocal tools like G\aa{}rding inequality. Such a strategy yields the exponential growth result thanks to a positive commutator type argument. Hence the main difficulty when proving Theorem~\ref{main.theorem} is of dynamical nature and it lies in the construction of this escape function. The key dynamical ingredient to deal with this problem is issued from the microlocal point of view that was introduced in~\cite{Faure.Sjostrand} to study the correlation (or Ruelle) spectrum of dynamical systems with hyperbolic behavior and that was further developped in~\cite{DyatlovZworski2016, DyatlovGuillarmou, DR.MS1, BonthonneauWeich, Faure.Tsujii, Meddane} among others. See also~\cite{BaladiTsujii, FaureRoySjostrand} for earlier contributions for discrete time dynamics using microlocal methods and~\cite{Lefeuvre} for a recent book describing these microlocal methods and some of their applications to dynamical systems. 

One of the key steps when considering these spectral problems is also the existence of an escape function for the lifted symplectic dynamics  or the use of the related radial estimates. In that spectral set-up, this leads to the construction of (microlocal) functional spaces adapted to the spectral analysis of operators like $\mathcal{L}_V$ and we will adapt these constructions to fit the PDE problem at hand in this article. We emphasize that, in view of our application to Theorem~\ref{theorem.torus}, the time dependent perturbations we consider are quite strong as they do not decay with time. This may lead to strong bifurcations compared with the unperturbed equation. This requires to pay attention when adapting the dynamical results from these references and, in particular, it seems that the anisotropic Sobolev norms from the above works cannot be used directly to deal with our problem except for $\e=0$ or for perturbations that would tend fast enough to $0$ as $t\rightarrow\infty$. To circumvent these restrictions, we will more specifically revisit the dynamical constructions from~\cite{DR.MS1} which deals with the Morse-Smale case. See the introduction of Section~\ref{section.escape} for a more detailed comparison. Modulo some extra work, the method presented below could probably be adapted to deal with open hyperbolic dynamical systems as in \cite{DyatlovGuillarmou} or with general Axiom A flows satisfying appropriate transversality assumptions as in~\cite{Meddane} (see also~\cite{Faure.Sjostrand} for the particular case of Anosov flows).

\begin{remark}
Prior to these microlocal methods involving the construction of an escape function, other (more geometric) functional spaces adapted to hyperbolic dynamics were developed in~\cite{Liverani, ButterleyLiverani, ButterleyLiverani2} and it is most likely that the functional spaces constructed there could also provide informations on the instabilities of linear PDEs such as~\eqref{transport.eq}.
\end{remark}

In fact, Theorem~\ref{main.theorem} holds true for more general time dependent perturbations of Equation~\eqref{transport.eq} and we refer to Equation~\eqref{pseudo.transport} for a more general class of equations (including unbounded pseudodifferential perturbations of~\eqref{transport.eq}) to which Theorem~\ref{main.theorem} applies -- see Theorem~\ref{main.theorem.pseudo}. As we shall see, the exponential growth of the solution can be obtained as soon as an explicit microlocal criterium~\eqref{e.criterium.growth} is satisfied. Roughly speaking, initial data for which the Sobolev norms will grow exponentially will be highly oscillating in the conormal direction to the stable manifold. Somehow equivalently, we will pick initial data whose semiclassical wavefront set~\cite[Section 8.4]{Zworski} intersects this direction in a nontrivial way.

Finally, the study of the growth of Sobolev norms for time dependent perturbations of linear partial differential equations is motivated by its potential applications to nonlinear PDE. It would be interesting to understand if the dynamical and microlocal method used in this article could be adapted to deal with such problems. In that direction, we mention the recent paper \cite{MasperoMurgante} in which the method of constructing an escape function is extended to a one dimensional fractional \emph{quasilinear Schrödinger} equation, allowing to construct solutions undergoing growth of Sobolev norms. Closely related to the dynamical approach of the present work, the other recent work~\cite{Vlasov} uses similar methods to study the exponential stability of \emph{nonlinear} Vlasov equations on negatively curved manifolds.

\subsection{Comparison with the Schr\"odinger case} 
We conclude this introduction by briefly discussing the related question of Sobolev instabilities for time dependent perturbations of Schr\"odinger equations,
\begin{equation}\label{schrodinger}
i\partial_t u = -\Delta_g u +Q(t) u,
\end{equation}
where $Q(t)$ is smooth and real valued on $(M,g)$. From the perspective of quantum mechanics, this is a classical problem~\cite[Section I.2]{Bellisard} and the growth of Sobolev norms can be viewed as an indicator of instabilities at the quantum level. Indeed, in this physical setting, a growth for the Sobolev norms of solutions to~\eqref{schrodinger} amounts to say that the quantum state does not stay in a bounded region of phase space. As explained in~\cite{Bellisard}, this can be understood as the quantum analogue that an Hamiltonian flow has unbounded trajectories. For what concerns linear, time-dependent Schrödinger equations on the torus, the problem has been first issued in \cite{BourgainGSN1999} when $Q(t)$ is periodic in time. In that setting, Sobolev norms of solutions are of order $\mathcal{O}(\langle t\rangle^\e)$ for every $\e>0$, showing that escape in phase space can only occur at a very slow rate. See also~\cite{Bourgain1999} in the quasi-periodic case. This kind of results was further generalized in~\cite{Delort2010, Wang, EliassonKuksin, MasperoRobert, BambusiGrebertMasperoRobert}. In particular,~\cite{Delort2010, MasperoRobert, BambusiGrebertMasperoRobert} rely among other things on the assumption that the eigenvalues of the principal part of the operator are organized into clusters with increasing spectral gaps. Without this asymptotic spectral gap property, it was proved that there are solutions to~\eqref{schrodinger} whose Sobolev norms grow polynomially fast~\cite{Delort2014}. It is worth noting that~\cite{Delort2014} rather considers the harmonic oscillator $-\partial_x^2+x^2$ on the real line (instead of $\Delta_g$). In that case, spectral gaps are of size $1$ like the operator $\nu\cdot \nabla_x$ appearing in~\eqref{transport.torus}. Observe also that perturbations in~~\cite{Delort2014} are of pseudodifferential type with order $0$. See also~\cite{Mas19, LiangZhaoZhou, Thomann21, Mas22, FaouRaphael, Mas23, LangellaMasperoRotolo25} for further generalizations of these results in various contexts related to perturbations of the harmonic oscillator. 

\paragraph{Organization of the article.} We now quickly describe the structure of the article. As already alluded, the core of our proof is the construction of an escape function for the transport operator in the Morse-Smale case. This requires a detailed description of the Hamiltonian flow induced by its principal symbol. Thus we devote Section \ref{MS.section} to a brief reminder of Morse-Smale dynamics and their symplectic lifts and then Section \ref{section.escape} to the construction of the corresponding escape function. In Section \ref{sec.pseudo}, we review tools from microlocal analysis and apply them to the proof of (a slightly stronger version of) Theorem~\ref{main.theorem} by a positive commutator argument making use of the escape function from Section \ref{section.escape}. Once this is done, we prove Theorem~\ref{theorem.torus} in Section~\ref{sec.2D} by combining Theorem~\ref{main.theorem} with a normal form procedure. Finally, Appendix~\ref{Appendix.dyn} reviews a few definitions from dynamical systems theory that are used all along the article (especially in Sections~\ref{MS.section} and~\ref{section.escape}) and in Appendix \ref{a:existence} we prove global well posedness for the class of transport equations that we use in Section \ref{ss.general.transport}.

\paragraph{Acknowledgements} Most of this work was carried our when the second author was visiting the Laboratoire de math\'ematiques Jean Leray (Nantes Universit\'e) with the financial support of the Erasmus+ Traineeship project. The support of these institutions was essential to the realization of this work. MTR acknowledges the support of the INdAM group GNAMPA. GR was also partially supported by the Institut Universitaire de France, by the Agence Nationale de la Recherche through the grant ADYCT (ANR-20-CE40-0017) and by the Centre Henri Lebesgue (ANR-11-LABX-0020-01). Both authors wish to thank B. Langella and A. Maspero for pointing out this problem and for the useful discussions during the preparation of this work.

\section{Review of dynamical setting}\label{MS.section}
In this section, we recall the definition of the class of Morse-Smale vector fields $V$ on a smooth compact manifold $M$ and of their symplectic lift to the cotangent bundle $T^*M$. Along the way, we collect some of their properties that will be used in the following. More precisely, in paragraph~\ref{ss:morse-smale}, we review the definition of such vector fields. After that, we briefly recall in~ paragraph~\ref{ss:symplectic-lift} how to lift a vector field in a symplectic manner to $T^*M$ and we review in paragraph~\ref{MS.lift} some material from~\cite{DR.MS1} on the dynamics of symplectic lifts for Morse-Smale vector fields. 

\begin{remark}
We refer to \cite[Chapter 4]{palis1998} for a detailed introduction to Morse-Smale vector fields and flows on smooth manifolds and to \cite{DR.MS1}  for a detailed study of the lifted dynamics on the cotangent bundle. We collect in Appendix \ref{Appendix.dyn} standard definitions from the theory of dynamical systems that are used in this section. 
\end{remark}

\subsection{Morse-Smale flows}\label{ss:morse-smale}
We begin with the definition of the vector fields of interest for our analysis.
\begin{definition}[Morse-Smale vector fields]\label{def.MS}
    We say that a smooth vector field $V \in \fX(M)$ is Morse-Smale if 
    \begin{enumerate}
        \item $V$ has a finite number of critical elements (i.e. critical points and closed orbits), which are all hyperbolic (see Definition \ref{hyp.crit}); 
        \item for any pair $\Lambda_1, \Lambda_2$ of critical elements, the stable manifold $W^s(\Lambda_1)$ and the unstable one $W^u(\Lambda_2)$ (see Definition \ref{s.u.man}) intersect transversally; 
        \item the nonwandering set (see Definition \ref{nonwand.set}) $\operatorname{NW}(V)$ coincides with the union of the critical elements of $V$. 
    \end{enumerate}
We denote by $\fX^{MS}(M) \subset \fX(M)$ the set of Morse-Smale vector fields on $M$.  
\end{definition}
These flows were introduced by Smale in~\cite{Smale1960} and they are generalizations of Morse-Smale gradient flows. They are the simplest examples of Axiom A flows as later defined by Smale in~\cite{Smale1967}. Given a vector field $V \in \fX(M)$, we denote by $\varphi^t_V(x)$ its flow on $M$ with starting point $x \in M$: 
\begin{equation}\label{flow}
\frac{\di}{\di t} \varphi^t_V(x)=V(\varphi^t_V(x)), \quad \forall \ t \in \R. 
\end{equation}
We recall that, from compactness of the manifold $M$, the flow $\varphi^t_V(x)$ is complete, i.e., for each initial datum $x \in M$, the flow line $\varphi^t_V(x)$ exists globally in time -- see e.g. \cite[Corollary 9.17]{lee2003}.  
We denote by 
\begin{equation}\label{critical.set}
    \mbox{Crit}(V):= \{ \Lambda_1, \ldots, \Lambda_K\}
\end{equation}
the set of critical elements (basic sets) of $V$, which are either critical points or closed orbits. It will be important for the final steps\footnote{More precisely, this is used when we pick the initial data generating unstable solutions.} of our analysis that $n\geq 2$ or that $\mbox{Crit}(V) \neq M$ if $n=1$.

From the Definition \ref{def.MS} of Morse-Smale vector fields, one can prove the following properties: 
\begin{lemma}
Let $V \in \fX^{MS}(M)$. 
\begin{itemize}
    \item[(i)] For every $x \in M$ there exists a unique pair $(i, j)_x \in \{1, \ldots, K\}^2$ such that 
    \begin{equation}
        x \in W^s(\Lambda_i) \cap W^u(\Lambda_j), 
    \end{equation}
    (recall notation in \eqref{critical.set} and Definition \ref{s.u.man}). 

    \item[(ii)] The unstable manifolds of the critical elements of $V$ form a partition of $M$, precisely: 
    \[
    M= \bigcup_{i=1}^K W^u(\Lambda_i), \quad \mbox{ and } \quad W^u(\Lambda_i) \cap W^u(\Lambda_j)= \emptyset, \ \forall \ i \neq j. 
    \]
\end{itemize}
\end{lemma}
We refer for instance to \cite[Lemmas 3.3, 3.4]{DR.MS1} for the proof of these statements.

\paragraph{Energy function.}
We now report a result by Meyer \cite{Meyer}, that proves the existence of a function which is non-decreasing along the flow that we call \emph{energy function} following the literature. It will be a fundamental element for our construction in Section \ref{section.escape}.
\begin{lemma}[\cite{Meyer} Energy function]\label{Lemma.energy}
    Let $V \in \fX^{MS}(M)$. Then there exists a function $\mathcal{E} \in C^\infty(M)$ such that 
    \begin{equation}\label{eq.energy}
        \cL_V \mathcal{E} \geq 0, \quad \mbox{ everywhere on M, and } \quad \cL_V \mathcal{E} >0 \quad \mbox{ on } M \setminus \operatorname{Crit}(V).
    \end{equation}
\end{lemma}

\begin{remark}
 In the case where $V$ is a gradient vector field associated with a  Morse function $f$, such a function is given by $f$ itself.
\end{remark}

\paragraph{Morse-Smale vector fields on surfaces.}
The case of Morse-Smale vector fields on compact surfaces (i.e. $n=2$) is peculiar and further motivates the choice of such class of vector fields. Indeed, we first remark that, in this case, the definition of Morse-Smale vector fields can be simplified (see for example \cite[Prop. 1.1, p.122]{palis1998}) as the second condition in Definition \ref{def.MS} can be replaced by 
\begin{itemize}
    \item[2.s] There are no trajectories connecting saddle points \footnote{We call saddle point an hyperbolic critical point $\Lambda$ in the sense of Definition \ref{hyp.crit}, such that the eigenvalues $\lambda_1, \lambda_2$ of $\di V(\Lambda)$ satisfy Re$(\lambda_1)>0$ and Re$(\lambda_2)<0$.}.  
\end{itemize}
In particular, condition 2.s makes the transversality condition of Definition \ref{def.MS} more explicit and easier to visualize. \begin{remark}\label{generic.2d}
    The set $\fX^{MS}(M)$ is generic in $\fX(M)$ (i.e., it is open and dense with respect to the $C^\infty$ topology), for $M$ smooth, compact manifold of dimension two. This important property of Morse-Smale vector fields has been proved by Peixoto in \cite[Theorem 2]{Peixoto1962}. We will use this result in Section \ref{sec.2D} in order to prove Theorem \ref{theorem.torus}. In higher dimension, it is still an open set which is contained in the larger class of Axiom A flows (with appropriate transversality properties). However, Axiom A flows are not generic in higher dimensions even if they are structurally stable. 
\end{remark}

\subsection{Symplectic lift on the cotangent bundle}\label{ss:symplectic-lift}
From now on, we will use the notation $T^*M\setminus \{ 0 \} := \{ (x, \xi) \in T^*M : \xi \neq 0 \}$
where $T^*M$ is the cotangent bundle of $M$. The main dynamical object considered in this work is the Hamiltonian flow of 
\begin{equation}\label{h}
    h: T^*M  \to \R, \quad h(x, \xi):= \xi (V(x)).
\end{equation}
It turns out that the dynamics of such flow is linked to that of the flow of $V(x)$ on $M$ and, in this paragraph, we make this connection precise. 

We denote by $\omega(\cdot, \cdot)$ the canonical symplectic form on $T^*M$ and by $\cX_h$ be the Hamiltonian vector field defined by $h$ (see \eqref{h}) through the relation
\begin{equation}\label{ham.vf}
    \di h (Y) = \omega(\cX_h, Y), \quad \forall \ Y \in \fX(M).
\end{equation} 
The upcoming definitions and results in this paragraph hold for any vector field, not necessarily Morse-Smale and we denote by $\pi:(x,\xi)\in T^*M\mapsto x\in M$ the canonical projection.

\begin{definition}[Symplectic lift]\label{def.lift}
    Let $V \in \fX(M)$ and let $\varphi^t(x)$ be a flow on $M$. 
    \begin{itemize}
        \item[I)] The symplectic lift of $V$ to $T^*M$ is the vector field $Y \in \fX(T^*M)$ such that 
        \begin{enumerate}
            \item[a)] $\di\pi(Y)=V$; 
            \item[b)] the vector field $Y$ satisfies $\cL_Y \omega =0.$
        \end{enumerate}
    
    \item[II)] The symplectic lift of a
    flow $\varphi^t$ on  $M$ is the flow $\Phi^t$ on $T^*M$ such that 
    \begin{enumerate}
        \item[a)] $\pi\circ \Phi^t=\varphi^t\circ\pi$, for all $t\in\mathbb{R}$, as maps on $T^*M$; 
        \item[b)]
        $
        \left(\Phi^t \right)^* \omega= \omega, \quad \forall t \in \R,
        $
    \end{enumerate}
    where $^*$ denotes the pullback of $\omega$ with respect to $\Phi^t$.
    \end{itemize}
\end{definition}
In order to write down explicitly the relation between a flow on $M$ and its symplectic lift on $T^*M$, we recall the following definition. 
\begin{definition}[Transposed operator]\label{def.transp}
  Let $\Phi: TM \to TM$ and set $\pi_0:(x,v)\in TM\mapsto x\in M$. We define the transposed map $\Phi^\top$ as 
  \begin{equation}
   \Phi^\top: T^*M \to T^*M, \quad \Phi^\top(\xi)(\mathrm{w}):=\xi(\Phi(\mathrm{w})), \quad \forall \xi \in {T^*_{\pi_0\circ\Phi(\mathrm{w})}M}, \ \mathrm{w} \in TM.  
  \end{equation}
\end{definition}

\noindent The following Lemma explicits the relation between a vector field $V(x)$ and the Hamiltonian vector field $\cX_h$ on $T^*M$, as defined in~\eqref{h} and \eqref{ham.vf}. We only state the next result, since  the proof is standard and follows directly applying Definitions \ref{def.lift} and \ref{def.transp}. 
\begin{lemma}\label{lemma.lift}
    For every $V \in \fX(M)$, the vector field $\cX_h$ defined by \eqref{ham.vf} is the symplectic lift of $V$ on $T^*M$. Moreover, for every flow $\varphi^t(x)$ on $M$, its symplectic lift is given by 
    \begin{equation}\label{flow.lift}
        \Phi^t(x, \xi)= \left(\varphi^t(x), \left[\di_x\varphi^t(x)\right]^{-\top} \xi \right),
    \end{equation}
    where $\left[\di_x\varphi^t(x)\right]^{-\top}$ is the inverse of the transposed operator of $\di_x \varphi^t(x)$, as in Definition \ref{def.transp}. 
\end{lemma}

In particular remark that, as an immediate consequence of Lemma \ref{lemma.lift}, the flow of the Hamiltonian vector field $\cX_h$ on $T^*M$, with $h$ in \eqref{h}, coincides with the lift of the flow $\varphi^t_V$. Thus we will use the notation $\Phi^t_V(x, \xi)$ for such flow, whose explicit expression is the one in \eqref{flow.lift}.

\subsection{Dynamics of Morse-Smale symplectic lifts}\label{MS.lift}
In this paragraph, we take $V \in \fX^{MS}(M)$ and we describe the dynamics of the Hamiltonian flow induced by the Hamiltonian vector field $\cX_h$, with $h: T^*M  \to \R$ defined in \eqref{h}. In particular, we describe global attractors and repellors for the lifted flow on $T^*M $ and some of their properties, using that $V$ is Morse-Smale.
We closely follow \cite{DR.MS1}, to which we refer for the proof of most results. 

Let $S^*M$ denote the unit cotangent bundle of $M$: 
\begin{equation}\label{S*M}
S^*M := \{ (x, \xi ) \in T^*M \ : \ \norm{\xi}_x= 1\},
\end{equation}
where $\norm{\cdot}_x$ is the norm on the cotangent space at the point $x$, $T^*_xM$,  induced by the metric $g_x$. 
The flow $\Phi^t_V$ of $\cX_h$ on $T^*M$, which is explicitly given in \eqref{flow.lift}, induces a projected flow on $S^*M$: 
\begin{equation}\label{proj.flow}
    \tilde{\Phi}_V^t(x,\xi):= \left( \varphi^t_V(x), \frac{\left[\di_x\varphi^t_V(x)\right]^{-\top} \xi}{\norm{ \left[\di_x\varphi^t_V(x)\right]^{-\top} \xi}_x}\right),
\end{equation}
and we denote by $\tilde{\cX}_h$ the induced smooth vector field on $S^*M$, which is the generator of $\tilde{\Phi}^t_V$.

We now identify global attractors and repellors for such flows on $S^*M$. To this aim, we start with the general definition of \emph{conormal bundle} of a submanifold. Let $S \subset M$ be a smooth submanifold of $M$. We define the submanifold $N^*S$ of $T^*M$ as 
\begin{equation}\label{conormal}
N^*S:= \{ (x, \xi) \in T^*M \mbox{ such that } x \in S, \ \xi \neq 0, \ \xi(\mathrm{w})=0, \ \forall\  \mathrm{w} \in T_xS\},
\end{equation}
and we call it the conormal bundle of $S$. We define the following sets: 
\begin{equation}\label{A.R}
    R_p:= \bigcup_{i=1}^K N^*(W^s(\Lambda_i)) \cap S^*M, \quad A_p:= \bigcup_{i=1}^K N^*(W^u(\Lambda_i)) \cap S^*M. 
\end{equation}
Next we recall that every stable and unstable manifold is foliated by a family of submanifolds (see Definition \ref{def.foliation}): 
\begin{equation}\label{uu.ss}
    W^u(\Lambda_i)= \bigcup_{x \in \Lambda_i} W^{uu}(x), \qquad W^s(\Lambda_i)= \bigcup_{x \in \Lambda_i} W^{ss}(x), \quad \forall 1 \leq i \leq K. 
\end{equation}
On the one hand, this foliation is trivial for critical points. On the other hand, for every $x \in \Lambda_i$ with $\Lambda_i$ closed orbit, $W^{ss}(x)$ has codimension one in $W^s(\Lambda_i)$ and analogously $W^{uu}(x)$ in $W^u(\Lambda_i)$. As we shall see, they also turn out to be relevant for the dynamics. Hence, we define 
\begin{equation}\label{A.a,R.r}
    R^r_p:= \bigcup_{i=1}^K \bigcup_{x \in \Lambda_i} N^*(W^{ss}(x)) \cap S^*M, \quad A^a_p:= \bigcup_{i=1}^K \bigcup_{x \in \Lambda_i} N^*(W^{uu}(x)) \cap S^*M. 
\end{equation}
\begin{remark}
    The transversality assumption in Definition \ref{def.MS} can be written explicitely as follows: for any $(i, j) \in \{ 1, \ldots, K\}^2$, for any $x \in W^s(\Lambda_i) \cap W^u(\Lambda_j)$ we have 
    \begin{equation}\label{transversality}
        T_xM=T_xW^s(\Lambda_i) + T_xW^u(\Lambda_j).
    \end{equation}
        Using this decomposition of $T_xM$, one can verify that transversality implies
    \begin{equation}\label{intersections}
        R_p \cap A_p= \emptyset, \quad R_p \cap A^a_p = \emptyset, \quad R^r_p \cap A_p= \emptyset. 
    \end{equation}
\end{remark}

 We are now in position to describe the dynamics of $\tilde{\Phi}^t_V(x, \xi)$  in \eqref{proj.flow} on $S^*M$. 
\begin{proposition}[\cite{DR.MS1}, Lemma 5.1, Theorem 5.2]\label{prop.proj.flow}
Let $V \in \fX^{MS}(M)$. Then, the sets $R_p$, $A_p$,  $R^r_p$, $A^a_p$ in \eqref{A.R} and \eqref{A.a,R.r} are compact subsets of $S^*M$. Moreover, denoting by $\bd$ the induced Riemannian distance on $S^*M$, the following holds: 
\begin{equation}
    \mbox{ for every }(x, \xi) \in S^*M \setminus A_p \quad \mbox{ we have } \quad \lim_{t \to -\infty} \bd(\tilde{\Phi}^t_V(x, \xi), R^r_p) =0, 
\end{equation}
\begin{equation}
    \mbox{ for every }(x, \xi) \in S^*M \setminus R^r_p \quad \mbox{ we have } \quad \lim_{t \to +\infty} \bd(\tilde{\Phi}^t_V(x, \xi), A_p) =0, 
\end{equation}
and 
\begin{equation}
    \mbox{ for every }(x, \xi) \in S^*M \setminus R_p \quad \mbox{ we have } \quad \lim_{t \to +\infty} \bd(\tilde{\Phi}^t_V(x, \xi), A^a_p) =0, 
\end{equation}
\begin{equation}
    \mbox{ for every }(x, \xi) \in S^*M \setminus A^a_p \quad \mbox{ we have } \quad \lim_{t \to -\infty} \bd(\tilde{\Phi}^t_V(x, \xi), R_p) =0. 
\end{equation}

\end{proposition}
We refer to \cite{DR.MS1} for the proof of this statement, which involves a careful study of the Morse-Smale structure of the flow induced by $V$ on $M$. Strictly speaking, the proof in that reference is given under the extra assumption that the vector field is $\mathcal{C}^1$-linearizable near critical elements. Yet, this hypothesis is removed by Meddane in~\cite[Lemma 3.2, Prop. 3.3]{Meddane} where he deals with general Axiom A flows satisfying the corresponding transversality property. We emphasize that one of the main technical issue in proving this statement is the compactness property. Finally, for later purpose, we define: 
\begin{equation}\label{A.R.TM}
    A:= \bigcup_{i=1}^K N^*(W^s(\Lambda_i)), \qquad R:=\bigcup_{i=1}^K N^*(W^u(\Lambda_i)).
\end{equation}
Analogously we define $A^a$ and $R^r$ as subsets of $T^*M \setminus \{ 0 \}$, obtained extending by homogeneity the projective ones in \eqref{A.a,R.r}.

\section{Construction of the escape function}\label{section.escape}
We are now in position to construct the main dynamical ingredient for the proof of Theorem \ref{main.theorem}: an \emph{escape (or Lyapunov) function} for the Hamiltonian $h(x,\xi)$ defined in \eqref{h}. Roughly speaking, we are looking for a function that increases along the flow lines of $\Phi_V^t$ in view of applying a positive commutators argument in our analysis of the transport equation. Before stating the main dynamical result we are aiming at, we start by giving a general definition of escape functions in our context.

\begin{definition}[Positively homogeneous function]\label{def.pos.homo}
    We say that a smooth function $f: T^*M \setminus \{ 0 \} \to \C $ is a positively homogeneous function of degree $\sigma \geq 0$ if 
    \[
    f(x, \lambda \xi)= \lambda^\sigma f(x, \xi), \quad \forall \lambda >0, \ (x, \xi) \neq 0.
    \]
\end{definition}

\begin{definition}[Escape function]\label{def.escape}
    Let $h: T^*M\to \R$ be the Hamiltonian function defined in~\eqref{h}. We say that a smooth function $a: T^*M \setminus \{ 0 \} \to \R$ is an escape function of order $\sigma > 0$ for $h$ if $a$ is positively homogeneous of degree $\sigma > 0$ and if there exist $\delta >0$ and a closed conic set $E \subset T^*M\setminus\{0\}$ such that
    \begin{equation}\label{bound.escape}
        \cX_h (a) \geq \delta \norm{\xi}_x^\sigma, \quad \mbox{ on } \{ (x, \xi) \in T^*M \setminus E \}, 
        \end{equation}
    \begin{equation}\label{bound.escape.2}
        \cX_h(a) \geq 0 \mbox{ and } a(x,\xi) \geq \delta \norm{\xi}^\sigma_x, \quad \mbox{ on } \{(x, \xi) \in E  \}, 
    \end{equation}
    where we recall the notation $\cX_h$ in \eqref{ham.vf} for the Hamiltonian vector field defined by $h$. 
\end{definition}

\begin{remark}
    We remark that, if $h$ is a positively homogeneous function of degree $\sigma_1$ and $a$ is a positively homogeneous function of degree $\sigma_2$, then $\cX_h(a)$ is a positively homogeneous function of degree $\sigma_1+\sigma_2-1$. In particular this means that, in Definition \ref{def.escape}, both $\cX_h(a)$ and $a$ are homogeneous of degree $\sigma$. 
\end{remark}

The main result of this section is the following Proposition:
\begin{proposition}\label{prop.escape}
    Let $V \in \fX^{MS}(M)$ be a Morse-Smale vector field on $M$ and let $h(x, \xi):= \xi(V(x))$. Then there exist a neighborhood $\mathcal{U}$ of $\operatorname{NW}(V)$ (see \ref{nonwand.set}), $c_0>0$ and $0< \sigma_0 <1$, depending on $V$ such that, for every $0<\sigma<\sigma_0$, $h$ admits an escape function $a$ of order $\sigma$ in the sense of Definition \ref{def.escape}, with $E$ being a closed conic neighborhood of $A^a \cap R^r\cap T^*\mathcal{U}$ and with $\delta= c_0\sigma$.  Moreover, if $n\geq 2$ or if $V$ has at least one critical point, there exists a nonempty open conic subset $C$ satisfying 
    \begin{equation}\label{negative.escape}
        a(x,\xi)\leq -\frac{c_0}{2}\|\xi\|_x^\sigma\ \text{on}\ \{(x,\xi)\in C \}.
    \end{equation}
\end{proposition}
\begin{remark}
Observe that when $V$ has no closed orbits, $A^a\cap R^r$ is empty and our proof will in fact show that $E$ is empty.
\end{remark}
In order to prove Proposition \ref{prop.escape}, we will use all the dynamical ingredients introduced in Section \ref{MS.section}. We will construct an escape function $a \in C^\infty(T^*M \setminus \{ 0 \})$ of the form 
\begin{equation}\label{form.escape}
    a(x, \xi) := m(x,\xi) f(x, \xi), \quad \forall (x, \xi) \in T^*M \setminus \{ 0 \},
\end{equation}
where $m \in C^{\infty}(T^*M \setminus \{0\})$ is a positively homogeneous function of degree zero (recall Definition \ref{def.pos.homo}) that we call \emph{order function } following the literature (see e.g. \cite{Faure.Sjostrand}) and $f(x, \xi)$ is a positively homogeneous function of degree $\sigma>0$, with $\sigma$ to be determined depending on the vector field. Our construction is close to the one appearing in the microlocal approach to Ruelle resonances in dynamical systems and we refer for instance to~\cite{Faure.Sjostrand, DR.MS1, Meddane} for presentations close to ours. The main difference with these references is that we  pick $f$ to be $\sigma$-homogeneous while these earlier works use symbols of order $\sigma$ for any $\sigma>0$. Roughly speaking, they use $\log(1+\|\xi\|_x^2)$ while we use $\|\xi\|_x^{\sigma}$. This is due to the different nature of the problems at hand: growth of Sobolev norms for time dependent vector fields in our case compared with the spectral study that is considered in these previous contributions.

We will first construct the function $m$ that will strictly decay away from the repellors and attractors of the lifted flow. A similar construction has been already achieved in~\cite{DR.MS1} in the case of Morse-Smale flows. In view of our application, the only thing we need to pay attention to, compared with this reference, is to specify the values of $m$ near the various subsets $A_p$, $A^a_p$, $R_p$ and $R_p^r$. This is the content of paragraph~\ref{ss.proof.order} and more specifically of Lemma~\ref{order.function}. Once this is done, we will construct in paragraph~\ref{ss.proof.f} the function $f$ and show that it has the expected properties (up to signs issues) near the attractors and repellors. We will also verify that its Poisson bracket with $h$ behaves nicely away from these sets. Once this is done, we will gather these two functions as in~\eqref{form.escape} and prove Proposition~\ref{prop.escape} in paragraph~\ref{ss.proof.escape}.

\subsection{Construction of the order function}\label{ss.proof.order}
In this section, we construct the first component of the escape function, the \emph{order function} $m \in C^\infty(T^*M \setminus \{ 0 \})$ in \eqref{form.escape}. Recall the definition of the sets in \eqref{A.R} and \eqref{A.a,R.r}.
Our first ingredient is the following lemma, which is an adaptation to our context of \cite[Lemma 2.1]{Faure.Sjostrand} (in which the authors deal with Anosov flows instead of Morse-Smale ones). 
\begin{lemma}\label{Lemma.m1}
    Let $V \in \fX^{MS}(M)$ and let $\tilde{\cX}_h$ be the induced vector field on $S^*M$ (see \ref{proj.flow}). For every $\e>0$ and for all neighborhoods $V(R_p)$ and $V(A_p^a)$ (of $R_p$ and $A_p^a$ respectively), there exist $\eta_1>0$, neighborhoods $W(R_p) \subset V(R_p)$, $W(A^a_p) \subset V(A^a_p)$ (of $R_p$ and $A_p^a$ respectively) and a function $m_1 \in C^\infty(S^*M; [0, 1])$ such that 
    \begin{align}
        & \tilde{\cX}_h(m_1) \geq 0 \quad \mbox{ on the whole } S^*M,\\
        &\label{property.m1}\tilde{\cX}_h(m_1) \geq \eta_1 \quad \mbox{on } S^*M \setminus (W(R_p) \cup W(A^a_p)).
    \end{align}
    Moreover $m_1<\e$ on $W(R_p)$ and $m_1>1-\e$ on $W(A^a_p)$. 
\end{lemma}
The proof of this Lemma was given in~\cite[Lemma 8.2]{DR.MS1} using the exact same argument as that of \cite[Lemma 2.1]{Faure.Sjostrand}. Note that, in order to follow this argument, one needs compactness of the sets $A^a_p$ and $R_p$ and existence of invariant neighborhoods of these sets for the flow -- see~\cite[Th.~5.4 and 7.1]{DR.MS1}. While this is obvious in the Anosov case, these two properties are much more subtle in the Morse-Smale case. Again, we emphasize that~\cite{DR.MS1} made the assumption that the flow is $\mathcal{C}^1$-linearizable and this restriction was removed by Meddane who performed this construction in the general case of Axiom A flows without any linearization assumption~\cite[Sect. 8]{Meddane}.

\begin{remark}\label{remark.m2}
    By applying the flow in backward time, the analog of Lemma \ref{Lemma.m1} holds for $R^r_p$ and $A_p$. For every $\e>0$ and for every $V(R^r_p)$ and $V(A_p)$ (neighborhoods of $R_p^r$ and $A_p$ respectively), there exist $\eta_2>0$, $W(R_p^r) \subset V(R^r_p)$ and $W(A_p) \subset V(A_p)$ (neighborhoods of $R_p^r$ and $A_p$ respectively) and $m_2 \in C^\infty(S^*M; [0,1])$ satisfying the analog of properties \eqref{property.m1} and such that $m_2< \e$ in $W(R^r_p)$ and $m_2>1-\e$ in $W(A_p)$. 
\end{remark}
\paragraph{Notations.} We denote by $W(R), \ W(R^r), W(A), \ W(A^a)$ the extensions by homogeneity to the whole $T^*M \setminus\{0\}$ of the sets in Lemma \ref{Lemma.m1} and Remark \ref{remark.m2}. We define
\begin{equation}\label{bad.region}
    \bB:= ( W(R) \cap W(R^r)) \cup (W(A) \cap W(A^a)) \cup (W(R^r) \cap W(A^a)). 
\end{equation}
Remark that, up to shrinking the $W$-neighborhoods, we can suppose that all the other intersections among such neighborhoods are empty (see \eqref{intersections}). Given a neighborhood $\mathcal{U}$ of $\operatorname{NW}(V)$ (see \ref{nonwand.set}), we also define 
\begin{equation}\label{Bu}
\bB_\mathcal{U}:=\bB\cap T^*\mathcal{U}.
\end{equation}
Recall that, from Definition \ref{def.MS}, $\operatorname{NW}(V)$ coincides with the set of critical elements of $V$.

We can now state the main result of this paragraph.
\begin{lemma}\label{order.function}
    Let $V \in \fX^{MS}(M)$, let $\cX_h$ be the induced Hamiltonian vector field on $T^*M \setminus \{ 0 \}$ (see \eqref{h}), let $\bB$ as in~\eqref{bad.region} and let $\mathcal{U}$ be a neighborhood of $\operatorname{NW}(V)$.  
    There exists a function $m \in C^\infty(T^*M \setminus \{ 0 \};[-4,4])$, positively homogeneous of degree zero, such that
    \begin{equation}\label{Xm.non.neg}
    \cX_h(m) \geq 0, \quad \mbox{ on the whole } T^*M\setminus \{ 0 \};
    \end{equation}
    and 
    \begin{equation}\label{m.outside}
        \cX_h(m) \geq \eta >0, \quad \mbox{ on } T^*M \setminus \bB_\mathcal{U}. 
    \end{equation}
    Moreover, 
    \begin{enumerate}
        \item[(i)] $m < -\frac12$ on $W(R^r) \cap W(R) $; 
        \item[(ii)] $m> \frac12$ on $W(A) \cap W(A^a)$; 
        \item[(iii)] $m >\frac14$ on $W(A^a) \cap W(R^r)$. 
    \end{enumerate}
\end{lemma}

\begin{proof}
    In order to prove the Lemma, we construct a function $\widetilde{m} \in C^\infty(S^*M; [-4,4])$ such that conditions \eqref{Xm.non.neg}, \eqref{m.outside} and items $(i)$, $(ii)$ and $(iii)$ hold for $\wt{m}$ and $\tilde{\cX}_h$ on $S^*M$. 
    Then, defining by homogeneity 
\begin{equation}\label{m.tilde.m}
    m(x, \xi) : = \widetilde{m}\left( x, \frac{\xi}{\norm{\xi}_x} \right), \qquad \forall (x, \xi) \in T^*M \setminus \{ 0\},
\end{equation}
we have $m \in C^\infty(T^*M \setminus \{ 0 \}; \R)$ and $\cX_h(m)= \tilde{\cX}_h(\tilde{m})$. 
Thus we now turn to defining the function $\wt{m}$ following the lines of~\cite{DR.MS1}. To this aim we recall the definition of the energy function $\mathcal{E}(x)$ in Lemma \ref{Lemma.energy} and we claim that, choosing $\e>0$ small enough in Lemma~\ref{Lemma.m1}, we can put 
\begin{equation}\label{tilde.m}
\tilde{m}\left(x, \frac{\xi}{\norm{\xi}_x}\right):= \mathcal{E}(x) -1 + \frac32 m_1\left(x, \frac{\xi}{\norm{\xi}_x}\right) + \frac12 m_2\left(x, \frac{\xi}{\norm{\xi}_x} \right),
\end{equation}
(we recall that $m_1$ and $m_2$ defined in Lemma \ref{Lemma.m1} and Remark \ref{remark.m2} depend on $\e$). First of all we remark that $\wt{m}$ is smooth since both $\mathcal{E}$, $m_1$ and $m_2$ are. Next we immediately see that $\tilde{\cX}_h(\wt{m})\geq 0$ since 
\[
\tilde{\cX}_h(\wt{m})\stackrel{\eqref{tilde.m}}{=}\underbrace{\tilde{\cX}_h(\mathcal{E})}_{\stackrel{\eqref{proj.flow}}{=}\cL_V(\mathcal{E}) \stackrel{\eqref{eq.energy}}{\geq}0} + \frac32\tilde{\cX}_h(m_1) + \frac12 \tilde{\cX}_h(m_2) \stackrel{\eqref{property.m1}}{\geq}0,
\]
proving the first  condition. Moreover, 
\[
\tilde{\cX}_h(\wt{m}) \stackrel{\eqref{eq.energy}}{=}\tilde{\cX}_h(E)+ \frac32\tilde{\cX}_h(m_1) + \frac12\tilde{\cX}_h(m_2) \geq \min \left\{ \min_{M\setminus\mathcal{U}}\mathcal{L}_V(E),\frac32\eta_1, \frac{\eta_2}{2} \right\}=: \eta>0 \qquad \mbox{outside } \bB_\mathcal{U},
\]
where we recall the definition of $\bB_{\cU}$ in \eqref{Bu} and we use, for the last inequality, condition \eqref{property.m1}, the analog in Remark \ref{remark.m2} for the function $m_2$ and Lemma~\ref{Lemma.energy} for $\mathcal{E}$. Thanks to \eqref{m.tilde.m}, this concludes the proof of \eqref{Xm.non.neg} and \eqref{m.outside}. 
We are left to prove items $(i)$, $(ii)$ and $(iii)$. To this aim we make the assumption 
\begin{equation}\label{E.1}
    \sup_{M} |\mathcal{E}| < \frac18,
\end{equation} 
that we will use along the proof of the three items. Notice that this is not in contradiction with previous part of the proof nor restrictive: indeed the function $\mathcal{E}$ is clearly bounded and, up to now, we have only used that $\cL_V(\mathcal{E})>0$, which is still true if we rescale $\mathcal{E}$ by a positive factor and impose \eqref{E.1}. 
\paragraph{Item $(i)$.} For this item, we have to look at the set $W(R^r_p) \cap W(R_p)$, where both $m_1<\e$ and $m_2 < \e$. Thus we have 
\[
\tilde{m} = \mathcal{E}(x) - 1 +\frac32 m_1 +\frac12 m_2 \stackrel{\eqref{E.1}}{\leq} -\frac78 + 2\e < -\frac 12, 
\]
choosing $0<\e< \frac{3}{16}$. From \eqref{m.tilde.m}, this concludes the proof of the first item.

\paragraph{Item $(ii)$.} For this item, we have to look at the set $W(A_p) \cap W(A^a_p)$, where both $m_1>1-\e$ and $m_2 >1- \e$. Thus we have 
\[
\tilde{m} = \mathcal{E}(x) - 1 +\frac32 m_1 +\frac12 m_2 \stackrel{\eqref{E.1}}{\geq} -\frac98 + 2(1-\e) =\frac78-2\e >\frac 12, 
\]
choosing again $0<\e< \frac{3}{16}$. From \eqref{m.tilde.m}, this concludes the proof of the second item.

\paragraph{Item $(iii)$.} For this item, we have to look at the set $W(A_p^a) \cap W(R^r_p)$, in which $m_1>1-\e$. We also use that $m_2\geq 0$. We have:
\[
\tilde{m} = \mathcal{E}(x) - 1 +\frac32 m_1 +\frac12 m_2 \stackrel{\eqref{E.1}}{\geq} -\frac98 + \frac32(1-\e) =\frac38-\frac32\e >\frac 14, 
\]
choosing $0 < \e < \frac{1}{12}$.\\
In conclusion choosing any $0<\e < \frac{1}{12}$ and the corresponding functions $m_1$ and $m_2$, all three items are satisfied, concluding the proof. 
\end{proof}

\subsection{Construction of the function $f$}\label{ss.proof.f}
In this section we construct the second component of the escape function, i.e. the function denoted by $f$ in \eqref{form.escape}. Precisely, we have the following Lemma.

\begin{lemma}\label{function.f}
    There exist conical neighborhoods $V(A)$, $V(A^a)$, $V(R)$ and $V(R^r)$ (respectively of $A$, $A^a$, $R$ and $R^r$) and a neighborhood $\mathcal{U}$ of $\operatorname{NW}(V)$ such that the following holds. 
    
    For any $\sigma>0$ one can find a function $f_{\sigma}: T^*M \setminus \{ 0 \} \to \R$, smooth, positively homogeneous of degree $\sigma$, with restriction to $S^*M$ positive, which satisfies the following properties: 
    \begin{itemize}
        \item[(i)] $\cX_h(f_{\sigma}) \leq - \gamma_1 \sigma f_{\sigma}, \quad $ in $V(R) \cap V(R^r)\cap T^*\mathcal{U}$;
        \item[(ii)]
        $\cX_h(f_{\sigma}) \geq \gamma_{2} \sigma f_{\sigma}, \quad $ in $V(A) \cap V(A^a)\cap T^*\mathcal{U}$; 
        \item[(iii)]
        $ f_{\sigma}(x,\xi)=|h(x, \xi)|^{\sigma}, \quad \mbox{ in } V(R^r) \cap V(A^a)\cap T^*\mathcal{U}$;
    \end{itemize}
    for some $\gamma_{1}, \gamma_{2}>0$ independent of $\sigma$. Moreover, there exists a constant $c_0>0$ such that, for every $0<\sigma\leq 1$,
    \begin{equation}\label{e.uniform.bound.sigma}
         c_0^{-1}\|\xi\|_x^\sigma\leq  f_{\sigma}(x,\xi)\leq c_0\|\xi\|_x^\sigma\ \text{and}\ |\cX_h(f_\sigma)(x,\xi)|\leq c_0\sigma \|\xi\|_x^\sigma, \quad \forall (x,\xi)\in T^*M\setminus\{0\}.
    \end{equation}
    
\end{lemma}

In order to prove this Lemma, we need the following result, which follows from the hyperbolic properties of the non-wandering set (see Definition \ref{def.MS}). 
\begin{lemma}\label{MS.growth}
    Let $V\in\fX^{MS}(M)$ and consider the associated Hamiltonian vector field $\cX_h$ with $h(x,\xi):=\xi(V(x))$.
    
    Then, there exist $\theta, \ C_a, C_r >0$ such that, for every $(x,\xi)\in T^*M$ verifying $x\in\operatorname{NW}(V)$,  
    \begin{equation}\label{exp.growth}
        \norm{[\di\varphi_V^{-t}(x)]^{-\top}\xi}_{\varphi^{-t}_V(x)} \leq C_a e^{-\theta t} \norm{\xi}_x, \quad \mbox{ in } A \cap A^a, \mbox{ and } \ \forall t >0, 
    \end{equation}
        and 
     \begin{equation}\label{exp.decay}
        \norm{[\di\varphi_V^t(x)]^{-\top}\xi}_{\varphi^t_V(x)} \leq C_r e^{-\theta t} \norm{\xi}_x, \quad \mbox{ in } R \cap R^r \mbox{ and } \  \forall t>0.
    \end{equation}
\end{lemma}

We are now ready to prove Lemma \ref{function.f}.

\begin{proof}[Proof of Lemma~\ref{function.f}]
Mimicking the proof of~\cite[Lemma 2.4]{BonthonneauWeich}, we first define the two functions $f^{r}_1$ and $f^a_1$ as 
\begin{equation}\label{f.sigma}
    f^{a/r}_1(x,\xi):= e^{ I^{a/r}(x, \xi)}, \quad \mbox{ with } I^{a/r}(x, \xi):=\begin{cases}
        T_1^{-1}\int_0^{T_1} \ln{\norm{[\di\varphi_V^{-t}(x)]^{-\top}\xi}_x}\di t, & \mbox{ in } V(A) \cap V(A^a)\\
        T_1^{-1}\int_0^{T_1} \ln{\norm{[\di\varphi_V^{t}(x)]^{-\top}\xi}_x}\di t, & \mbox{ in } V(R) \cap V(R^r),\\
    \end{cases}
\end{equation}
for $T_1>0$ to be chosen big enough and for small neighborhoods $V(\cdot)$. 
First of all we remark that the functions $f^{a/r}_1$ are positively homogeneous of degree $1$ in their domains. Indeed, via a direct computation one obtains 
\begin{equation}\label{homo.f}
I^{a/r}(x, \lambda \xi)= \ln{\lambda} + I^{a/r}(x, \xi) \quad \implies \quad f^{a/r}_1(x, \lambda \xi) = \lambda f^{a/r}_1(x,\xi), \quad \forall \lambda >0, (x, \xi) \neq 0. 
\end{equation}
Next we claim that, up to considering slightly smaller neighborhoods, the functions $f^{a/r}_\sigma:=(f^{a/r}_1)^\sigma$ verify the inequalities of items $(i)$ and $(ii)$ of the Lemma. Let us postpone this verification and conclude the proof extending $f_1^{a/r}$ to the whole $T^*M \setminus \{ 0 \}$. To this aim, we exploit homogeneity \eqref{homo.f} and construct a function $\tilde{f}_1 \in C^\infty(S^*M)$ such that 
\begin{equation}
\tilde{f}_1=
    \begin{cases}
        f^a_1 & \mbox{ in } V(A_p) \cap V(A^a_p)\\
        f^r_1 & \mbox{ in } V(R_p) \cap V(R^r_p) \\
        \left|h\left(x, \frac{\xi}{\norm{\xi}_x}\right)\right| & \mbox{ in } V(R^r_p) \cap V(A^a_p).
    \end{cases}
\end{equation}
Recall that, from \eqref{intersections}, the three sets $V(R) \cap V(R^r)$, $V(A)\cap V(A^a)$ and $V(A^a) \cap V(R^r)$ do not intersect (up to restricting the neighborhoods if needed). Moreover the function $\tilde{f}_1$ is positive in the closed set $\overline{\bB}$ defined by \ref{bad.region}. It is then possible to extend $\tilde{f}_1$ to a positive smooth function on the whole compact manifold $S^*M$ (see for example \cite[2.26]{lee2003}). Note that, by compactness, the extended function and its derivatives are uniformly bounded. Finally we extend $\tilde{f}_1$ by homogeneity to the whole $T^*M \setminus \{ 0 \}$ and obtain $f_{1}$ positively homogeneous of degree $1$. We define then $f_\sigma:=(f_1)^\sigma$ so that~\eqref{e.uniform.bound.sigma} is automatically satisfied using the fact that $f_1$ is $1$-homogeneous.

We are thus left to prove the initial claim, i.e., that $f^a_{1}$ and $f^r_{1}$ in \eqref{f.sigma} verify items $(i)$ and $(ii)$ of the Lemma (the bound for $\sigma>0$ follows again directly by construction). We show only item $(i)$ since the two calculations are exactly the same. We first prove the claim on $R \cap R^r\cap\text{NW}(V)$. 
To this aim, we first we evaluate $\cX_h(f_1^r)$:
\begin{align}\label{X.fr}
    \cX_h(f_1^r)=\frac{\di}{\di s}\Big|_{s=0}f_1^r(\Phi^{s}_V(x, \xi)) \stackrel{\eqref{f.sigma}}{=} f_1^r(x, \xi) \frac{\di}{\di s}\Big|_{s=0}I^r(\Phi^{s}_V(x,\xi)),
\end{align}
and, using Lemma~\ref{MS.growth}, we have, on $R\cap R^r\cap\text{NW}(V)$: 
\begin{align}\label{I.r}
    \frac{\di}{\di s}\Big|_{s=0}I^r(\Phi^{s}_V(x,\xi))
    &= \frac{1}{T_1}\int_0^{T_1}\frac{\di}{\di s} \ln(\norm{[\di\varphi_V^{t+s}(x)]^{-\top}\xi}) \Big|_{s=0}\di t\\ 
    &= \frac{1}{T_1}\int_0^{T_1} \frac{\di}{\di t} \left( \ln(\norm{[\di\varphi_V^{t}(x)]^{-\top}\xi})\right) \di t\\
    &= \frac{1}{T_1} \left( \ln\left(\norm{[\di\varphi_V^{T_1}(x)]^{-\top}\xi}\right) - \ln \norm{\xi}\right)\\
    & \stackrel{\eqref{exp.decay}}{\leq}\frac{1}{T_1}  \left( \ln(C_r) - \theta T_1 \right).
\end{align}
Thus, choosing $T_1$ big enough so that $\ln(C_r)- \theta T_1 < 0$ and plugging \eqref{I.r} in \eqref{X.fr} we have
\begin{equation}
    \cX_h(f^r_1) \leq  f^r_1 \left(\frac{\ln(C_r)}{T_1}- \theta\right)  < - f^r_1 \gamma_1, 
\end{equation}
with $\gamma_1 = -\frac{\ln(C_r)}{T_1}+ \theta  >0$. This proves item $(i)$ for $f^r_1$ on $R \cap R^r\cap\text{NW}(V)$.
Then we can extend by continuity the inequality to a sufficiently small conical neighborhood $\mathbf{C}$ of $R\cap R^r\cap\text{NW}(V)$. Recall that we aim at showing such an upper bound inside $V(R)\cap V(R^r)\cap T^*\mathcal{U}$. Up to shrinking the size of the neighborhood $\mathcal{U}$ of $\text{NW}(V)$ and of the conical neighborhoods $V(R)\cap V(R^r)$ and recalling from Proposition~\ref{proj.flow} that $R^r_p$ and  $R_p$ are both compact subsets, one finds that $\mathbf{C}$ contains $V(R)\cap V(R^r)\cap T^*\mathcal{U}$ in its interior, proving item $(i)$.  One analogously proves item $(ii)$ using \eqref{exp.growth}.
\end{proof}

\subsection{Proof of Proposition \ref{prop.escape}}\label{ss.proof.escape}
We now put together Lemmas \ref{order.function} and \ref{function.f} to prove Proposition \ref{prop.escape}. Let $m \in C^\infty(T^*M \setminus \{ 0 \} )$ and  $\{f_{\sigma} \in C^\infty(T^*M\setminus \{ 0 \})\}_{ 0<\sigma <1}$ be respectively the order function given by Lemma \ref{order.function} and the family of functions given by Lemma \ref{function.f}, for the vector field $V \in \fX^{MS}(M)$. We note that the neighborhoods $\mathcal{U}$ and $W(S)$ in Lemma~\ref{order.function} are chosen in terms of the neighborhoods $\mathcal{U}$ and $V(S)$ given by Lemma~\ref{function.f}.

 We will prove that it is possible to choose $\sigma >0$ small enough, such that the function 
\begin{equation}\label{a}
    a(x, \xi):= m(x, \xi) f_{\sigma}(x,\xi), \quad \forall (x, \xi) \in T^*M \setminus \{ 0 \}, 
\end{equation}
satisfies Definition \ref{def.escape} for some $\delta >0$ and $E$ a closed conical neighborhood containing $A^a \cap R^r$ in its interior (where we recall the definition of the sets $A^a$ and $R^r$ in \eqref{A.a,R.r}). Note that~\eqref{negative.escape} is satisfied in $W(R)\cap W(R^r)$ thanks to item (i) in Lemma~\ref{order.function} and to~\eqref{e.uniform.bound.sigma}. 
We now use the explicit expression of  $\cX_h(a)$ to show that $a$ satisfies the requirements of Definition~\ref{def.escape}: 
\begin{equation}\label{X.a}
    \cX_h(a) \stackrel{\eqref{a}}{=} \cX_h(m f_{\sigma}) = \cX_h(m) f_{\sigma} + \cX_h(f_{\sigma})m.
\end{equation}
Next, recalling the notation in \eqref{Bu}, we split the cotangent space in $T^*M = (T^*M \setminus \bB_\mathcal{U}) \cup \bB_\mathcal{U}$ and treat the two regions separately. Remark that $W(A^a) \cap W(R^r)\cap T^*\mathcal{U} \subset \bB_\mathcal{U}$ is one of the three components of the set $\bB_\mathcal{U}$. Let us start with $T^*M \setminus \bB_\mathcal{U}$, for which we have to prove \eqref{bound.escape}. Using non-negativity of $f_{\sigma}$ and equations \eqref{m.outside} and ~\eqref{e.uniform.bound.sigma}, we have: 
\begin{equation}\label{bound.1}
    \cX_h(a)= \cX_h(m) f_{\sigma} + \cX_h(f_{\sigma})m \stackrel{\eqref{m.outside}}{\geq} \eta f_{\sigma} + \cX_h(f_\sigma) m \stackrel{\eqref{e.uniform.bound.sigma}}{\geq} (\eta c_0^{-1} - \sigma c_0)\norm{\xi}_x^\sigma.
\end{equation}
Thus choosing $\sigma$ small enough (precisely $0<\sigma< \eta/(2c_0^2)$), we can conclude that
\begin{equation}\label{delta.1}
    \cX_h(a) \stackrel{\eqref{bound.1}}{\geq} \frac{\eta}{2c_0}\norm{\xi}^\sigma_x, \quad \mbox{ on } T^*M \setminus \bB_{\cU}, 
\end{equation}
proving \eqref{bound.escape} in this set. Remark that it is not restrictive to suppose that $\sigma<1$, as in the statement of Proposition \ref{prop.escape}.  

We now consider the three components of the region $\bB_\mathcal{U}$ and use the properties of $m$ and $f_\sigma$ in each of them (see Lemma \ref{function.f} and Lemma \ref{order.function}). 
    \begin{itemize}
        \item \b{In the set $W(R) \cap W(R^r)\cap T^*\mathcal{U}$}, \eqref{bound.escape} is given by: 
        \begin{equation}
            \cX_h(a)=\cX_h(m) f_{\sigma} + \cX_h(f_{\sigma}) m \stackrel{\eqref{Xm.non.neg}}{\geq} \cX_h(f_{\sigma}) m \geq \frac{\gamma_1 \sigma}{2}f_{\sigma},  
        \end{equation}
        where for the last inequality we have used point $(i)$ of Lemma \ref{order.function} and point $(i)$ of Lemma \ref{function.f}. 
        This concludes, for this set, the proof of \eqref{bound.escape}. Using~\eqref{e.uniform.bound.sigma}, we have
        \begin{equation}\label{delta.2}
            \frac{\gamma_1 \sigma}{2} f_\sigma \geq \frac{\gamma_1 \sigma}{2c_0} \norm{\xi}_x^\sigma, \quad \implies \quad \cX_h(a) \geq \frac{\gamma_1 \sigma}{2c_0} \norm{\xi}_x^\sigma,  \quad \forall (x, \xi) \in W(R) \cap W(R^r)\cap T^*\mathcal{U}. 
        \end{equation}
        Remark that in this set we have not used any smallness of $\sigma$ and that this bound is valid for any $0<\sigma\leq 1$. 
        \item \b{In the set $W(A) \cap W(A^a)\cap T^*\mathcal{U}$} analogously, \eqref{bound.escape} is obtained from: 
        \begin{equation}
            \cX_h(a)=\cX_h(m) f_{\sigma} + \cX_h(f_{\sigma}) m \stackrel{\eqref{Xm.non.neg}}{\geq} \cX_h(f_{\sigma}) m \geq \frac{\gamma_2 \sigma}{2}f_{\sigma},  
        \end{equation}
         where for the last inequality we have used point $(ii)$ of Lemma \ref{order.function} and point $(ii)$ of Lemma \ref{function.f}. This concludes the proof as in previous point using~\eqref{e.uniform.bound.sigma}:
         \begin{equation}\label{delta.3}
             \cX_h(a) \geq \frac{\gamma_1 \sigma}{2} f_{\sigma} \geq \frac{\gamma_2 \sigma}{2c_0} \norm{\xi}_x^\sigma.
         \end{equation}
                 
        \item \b{In the set $W(R^r) \cap W(A^a)\cap T^*\mathcal{U}$}, we instead have to prove \eqref{bound.escape.2}. First of all remark that, from point $(iii)$ of Lemma \ref{function.f}, we have that, for every $\sigma >0$, $f_{\sigma} \equiv |h(x, \xi)|^\sigma$, thus in particular
        \begin{equation}\label{X.f}
            \cX_h(f_{\sigma})=0.
        \end{equation}
        Plugging this expression in \eqref{X.a} we obtain 
        \begin{equation}\label{almost.there}
            \cX_h(a) \stackrel{\eqref{X.a}}{=} \cX_h(m) f_{\sigma} + \cX_h(f_{\sigma}) m \stackrel{\eqref{X.f}}{=} \cX_h(m) f_\sigma \geq 0,
        \end{equation}
        from \eqref{Xm.non.neg} and the fact that $f_\sigma$ is non negative. 
        This concludes the proof, since from point $(iii)$ of Lemma \ref{order.function} $m(x, \xi)>\frac14$ in this region, thus, using~\eqref{e.uniform.bound.sigma} one more time
        \begin{equation}\label{delta.4}
            a\stackrel{\eqref{a}}{\geq}\frac14 f_\sigma \geq \frac{1}{4c_0} \norm{\xi}_x^\sigma. 
        \end{equation}
    \end{itemize}

In conclusion, we obtain, for every $\sigma>0$ small enough, \eqref{bound.escape} and \eqref{bound.escape.2} with $\delta:= \frac{\sigma}{2c_0}\min \{ \eta, \gamma_1, \gamma_2, 1/2 \}$ (see \eqref{delta.1}, \eqref{delta.2}, \eqref{delta.3} and \eqref{delta.4}).

\section{Instability and growth of Sobolev norms}
We are now in position to prove Theorem \ref{main.theorem}. The two main ingredients we use are the existence of an escape function for the Hamiltonian $h(x, \xi)= \xi(V(x))$ with $V$ a Morse-Smale vector field on $M$ (proven in Proposition \ref{prop.escape}) and some results of symbolic calculus on $M$. For this reason, before proving Theorem \ref{main.theorem}, we will start this section by reviewing some preliminary material on pseudodifferential operators in paragraph \ref{ss.pseudo}. We refer to \cite[Ch. 14]{Zworski} and~\cite[App. E]{DyatlovZworski} for an introduction to pseudodifferential calculus on compact manifolds. See also \cite[Appendix A2, A3]{Vlasov} for a brief presentation close to ours.

Once these preliminary tools are settled, we introduce a generalized transport equation~\eqref{pseudo.transport} in paragraph~\ref{ss.general.transport} for which we will prove existence of solutions whose positive Sobolev norms grow exponentially. This is achieved in two steps. We introduce first in paragraph~\ref{ss.energy.estimate} an energy type functional that is bounded from above by the Sobolev norm of the solution and we prove a growth property~\eqref{growth.A} for this functional. Then, in paragraph~\ref{ss.initial.data}, we introduce appropriate initial data such that this functional indeed blows up exponentially in time along the corresponding trajectories. This concludes the proof of Theorem~\ref{main.theorem} in the case of the generalized transport equation~\eqref{pseudo.transport}.

\subsection{Pseudodifferential operators}\label{sec.pseudo}

\paragraph{Pseudodifferential operators in $\R^n$.} We start by giving here some results of symbolic calculus in the case of pseudodifferential operators on $\R^n$, as the case of manifolds will be locally modeled on it. First, we recall the definition of the set of symbols of order $\rho \in \R$ over $\R^n$: 
\begin{equation}\label{symbolsRn}
    \bS^{\rho}(\R^{2n}):=\{ a \in C^\infty(\R^{2n}):  \forall \alpha, \beta \in \N^n \ |\pa_x^{\alpha}\pa_{\xi}^{\beta} a(x, \xi) | \leq C_{\alpha, \beta} \la \xi \ra ^{\rho -|\beta|}\}
\end{equation}
(where we denote by $\la \xi \ra:= (1 + |\xi|^2)^{\frac12}$), and the definition of the Weyl quantization of symbols $a \in \bS^{\rho}(\R^{2n})$:
\begin{equation}\label{Weyl}
    \Opw(a)[u]:= \frac{1}{(2\pi)^n} \int_{\R^n} \int_{\R^n}e^{\im(x-y)\xi} a\left(\frac{x+y}{2}, \xi\right)u(y) \di y \di \xi.
\end{equation}
This last definition makes sense apriori only for $a$ and $u$ belonging to the Schwartz class but one can show that this still makes sense for $a\in\bS^{\rho}(\R^{2n})$ and $u$ in the Schwartz class. In this setting, we say that an operator $F$ belongs to $\Psi^\rho(\R^n)$, the class of pseudodifferential operators of order $\rho$, if there exists a symbol $f \in \bS^\rho(\R^{2n})$ such that $F=\Opw(f)$. The following general results for pseudodifferential calculus over $\R^n$ hold. We refer for example to \cite{Shubin, Zworski} for proofs. 
\begin{lemma}\label{symbolic.calculus}
    Let $f \in \bS^{\rho_1}(\R^{2n})$, $g \in \bS^{\rho_2}(\R^{2n})$, $\rho_1, \rho_2 \in \R$. Then:
    \begin{enumerate}
        \item For any $s \in \R$, there exist $C_{f, s}>0$ such that: 
        \begin{equation}\label{e.cv}
          \|\Opw(f)u\|_{s-\rho_1} \leq C_{f, s}\|u\|_s,
        \end{equation}  
        where we recall the definition of Sobolev norms in \eqref{sobolev.spaces} (with $M \leadsto \R^n$); 
        \item\label{aggiunto_Weyl}
        $\Opw(f)^*=\Opw(\overline{f})$ (here $^*$ denotes the adjoint operator) and in particular skew adjoint operators have purely imaginary symbols;
        \item There exists $h \in \bS^{\rho_1+\rho_2}(\R^{2n})$ such that 
        $\Opw(f)\circ \Opw(g)=\Opw(h).$ Furthermore,
        \begin{equation}\label{comp.symbols}
            h:= f g -\frac{i}{2} \{ f, g\} \ \operatorname{mod}\ \bS^{\rho_1+\rho_2-2}(\R^{2n}); 
        \end{equation}
        \item 
        There exists $\ell  \in \bS^{\rho_1+\rho_2-1}(\R^{2n})$ such that 
        $ \im [\Opw(f),\Opw(g)] = \Opw(\ell)$. 
        Moreover
        \begin{equation}\label{comm.symbols}
        \ell:= \{ f, g\}\ \operatorname{mod}\ \bS^{\rho_1+ \rho_2-2}(\R^{2n});
        \end{equation}
        \item
        If $\rho_1<0$, then $\Opw(f)$ is compact as an operator from $L^2(\R^n)$ to itself. 
    \end{enumerate}
\end{lemma}

We conclude this introductory section with the following observation, that is a crucial ingredient to consider pseudodifferential operators on manifolds. Given a smooth diffeomorphism $\gamma: \R^n \to \R^n$, with bounded derivatives and such that also $\gamma^{-1}$ is bounded with bounded derivatives, one can define its symplectic lift as 
    \begin{equation}\label{symp.lift}
        \wt{\gamma}: T^*\R^n\simeq\R^{2n} \to T^*\R^n, \ (x, \xi) \to \left(\gamma^{-1}(x), \di \gamma(x)^\top \xi \right).
    \end{equation}
It is possible to prove that if $a \in \bS^\rho(\R^{2n})$, then $a \circ \wt{\gamma} \in \bS^\rho(\R^{2n})$, see for example \cite[Theorem 9.4]{Zworski}.

\paragraph{Pseudodifferential operators on a compact manifold M.}\label{ss.pseudo} 
Following \cite{Zworski}, we first fix a finite atlas for $M$, 
\begin{equation}\label{atlas}
    \{(\cU_i, \gamma_i)\}_{i=1}^N, \qquad \bigcup_{i=1}^N \cU_i = M, 
\end{equation}
where $\{\cU_i\}_{i=1}^N$ are coordinate patches and $\gamma_i: \cU_i \to V_i \subset \R^n$ are smooth homeomorphisms onto open subsets $\{V_i\}_{i=1}^N$ of $\R^n$.
We are now in position to introduce symbols on the compact manifold $M$. This definition is given for example in \cite[Sections 14.2.2, 14.2.3]{Zworski}, and the idea is that a function $a \in C^\infty(T^*M)$ is a symbol over $M$ if, in local coordinates, it corresponds to a symbol on each open subset $\gamma_i(\cU_i)$ over $\R^n$. 
\begin{definition}[Symbols on $M$]\label{def.symbols}
    We say that a smooth function $a \in C^\infty(T^*M)$ is a symbol of order $\rho$ on $M$ if, for any $i=1, \ldots, N$, the pullback of $a$ under the identification~\eqref{symp.lift} 
    $$
    V_i \times \R^n \to T^*\cU_i
    $$
    belongs to $\bS^{\rho}(T^*V_i)$. 
    Equivalently we can define the set of symbols of order $\rho \in \R$ on $M$ as:
    \begin{equation}\label{set.symbols.M}
        \bS^\rho(T^*M):=\{ a \in C^{\infty}(T^*M) : \forall \alpha, \beta \in \N^{n} \ |\pa_x^\alpha \pa_\xi^\beta a(x, \xi) | < C_{\alpha, \beta} \la \xi \ra ^{\rho - |\beta|} \},
    \end{equation}
 where the derivatives are again understood in local coordinate charts.    
\end{definition}
We remark that Definition \ref{def.symbols} does not depend on the choice of the atlas in \eqref{atlas}. We record the following property:
\begin{lemma}\label{symbols.compact}
    Let $a \in C^\infty(T^*M)$ be a compactly supported function, then $a \in \bS^\rho(T^*M)$ for any $\rho<0$.
\end{lemma}

Following~\cite{Zworski}, we now set

\begin{definition}[Pseudodifferential operators of order $\rho$]\label{def.pseudodiff}
    We say that a linear operator $A: C^\infty(M) \to C^\infty(M)$ is a pseudodifferential operator of order $\rho$ on $M$ if the following two conditions hold: 
    \begin{enumerate}
        \item for every coordinate patch $\cU_i$ there exists a symbol $a_i \in \bS^{\rho}(\R^{2n})$ such that 
    \begin{equation}
        \varphi A(\psi u) = \varphi \gamma_i^* \Opw(a_i)(\gamma^{-1}_i)^*(\psi u), \quad \forall \varphi, \psi \in C_c^{\infty}(\cU_i), \ \forall u \in C^\infty(M),
    \end{equation}
    where $\gamma_i$ is the homemorphism associated to $\cU_i$ (see \eqref{atlas}), $^*$ denotes the pullback operator and $\Opw(a_i)$ is defined as in \eqref{Weyl};
    \item for every $\varphi_1, \varphi_2 \in C^\infty(M)$ such that $\mbox{supp}(\varphi_1) \cap \mbox{supp}(\varphi_2) = \emptyset$, the operator 
        $\varphi_1 A \varphi_2$ is bounded from $\cH^{-k}(M)$ to $\cH^k(M)$ for every $k \in \N$ (recall \eqref{sobolev.spaces}).  
    \end{enumerate}
We denote by $\Psi^{\rho}(M)$ the class of pseudodifferential operators of order $\rho$ on $M$. 
\end{definition}

It is non-trivial to associate to each symbol in $\bS^\rho(T^*M)$ an operator in $\Psi^\rho(M)$ and vice versa (see \cite[Theorem 14.1]{Zworski}). Following~\cite[Sect. 4]{Bonthonneau} (see also~\cite[App. A]{Vlasov}) and in view of simplifying some aspects of the discussion, we now assume that the atlas in \eqref{atlas} is an isochore atlas in the sense of the following definition:
\begin{definition}[Isochore atlas]\label{def.isochore}
    We say that the atlas $\{(\cU_i, \gamma_i)\}_{i=1}^N$ is \emph{isochore} if
    \begin{equation}
        \gamma_i^*(\di \mu_g)= \mbox{Leb}_{\R^n}, \quad \forall \ i=1, \ldots, N,
    \end{equation}
    where $\di \mu_g$ is the Riemmanian volume form on $M$ and Leb is the Lebesgue measure.  
\end{definition}
Roughly speaking, Definition \ref{def.isochore} 
can be rephrased saying that an atlas is isochore if, for every change of coordinates between two patches $\cU_i$ and $\cU_j$, the determinant of the Jacobian matrix related to such transformation is equal to one. It is a general result that there always exists an isochore atlas on a compact smooth and oriented Riemmanian manifold (see \cite{Moser}). Next, we fix a partition of unity $\{\chi_i\}_{i=1}^N$ subordinated to the atlas, i.e., 
\begin{equation}
    \sum_{i=1}^N\chi_i^2(x)=1, \quad \forall x \in M, \ \mbox{ and } \ \chi_i \in C^\infty_c(\cU_i; [0, 1]), \, \forall i=1, \ldots, N. 
\end{equation}
Finally, following~\cite[Ch. 14]{Zworski}, we can define the quantization of a symbol $a \in \bS^\rho(T^*M)$:
\begin{equation}\label{quant.manifold}
    \OpM(a):= \sum_{i=1}^N \chi_i \gamma_i^* \Opw((\wt\gamma_i^{-1})^*(\wt\chi_i a))(\gamma_i^{-1})^*\chi_i,
\end{equation}
where we recall the notation in \eqref{symp.lift} for $\wt{\gamma_i}$ and where $\wt\chi_i\in C^\infty_c(\cU_i; [0, 1])$ is such that $\wt\chi_i\chi_i=\chi_i$.

\begin{remark}
    Let us briefly motivate the expression in \eqref{quant.manifold}. Given a symbol $a \in \bS^\rho(T^*M)$, in order to define its quantization we need to use its expression in local charts to rewrite it as sum of symbols over $\R^n$. Indeed, each $\{(\wt{\gamma}_i^{-1})^*(\chi_i a)\}_{i=1}^N$ is a symbol over $\R^{2n}$. Thus we can quantize it using the Weyl quantization over $\R^{n}$ (see \eqref{Weyl}), and then go back to the manifold localizing in the desired chart, thanks to the composition with $(\gamma_i^{-1})^*\chi_i$ and its inverse. 
\end{remark}

We now state the main result related to this choice of the quantization and of an isochore atlas: roughly speaking, the quantization of a symbol is independent on the choice of charts, up to lower order remainders. Precisely, we have the following result (see \cite{Bonthonneau} or \cite[Remark A.2]{Vlasov}, to which we refer for a proof): 

\begin{lemma}
    Let $a \in \bS^\rho(T^*M)$ be a symbol of order $\rho$ over $M$ and let $\{( \cU_i, \gamma_i)\}_{i=1}^N$ be an isochore atlas in the sense of Definition \ref{def.isochore}. Let $\OpM(a)$ be the quantization of $a$ in the sense of \eqref{quant.manifold}. Then the principal symbol of $\OpM(a)$ is well defined and independent of the chart, up to symbols in $\bS^{\rho-2}(M)$. 
\end{lemma}
 We also record the following useful properties (see~\cite[Remark A.3]{Vlasov} for a proof of the second point):
\begin{lemma}\label{order.one.symbol}
\begin{enumerate}
    \item Let $f \in C^\infty(M)$ be a smooth function on $M$. Then $f \in \bS^0(T^*M)$ and  $\OpM(f) \in \Psi^0(M)$ is a multiplicative operator with factor $f$, i.e.,
    \[
    \OpM(f)u(x):= f(x) u(x), \quad \forall x \in M, \, \forall u \in C^\infty(M).
    \]
    \item Let $X \in \fX(M)$ be a smooth vector field on $M$, then 
    \begin{equation}\label{vf.quant}
        \cL_Xu(x)=\OpM(i \xi (X(x)) + r(x))u(x), \quad \forall u \in C^\infty(M),
    \end{equation}
    where $r \in \bS^{0}(T^*M)$ is a smooth function on $M$.  
\end{enumerate}
\end{lemma}
We now state the main properties of pseudodifferential operators that we will need in our proof. We proceed in analogy with the ones of Lemma \ref{symbolic.calculus} on $\R^n$. 
\begin{lemma}\label{manifold.symbolic}
    Let $a \in \bS^{\rho_1}(T^*M)$ and $b \in \bS^{\rho_2}(T^*M)$, with $\rho_1, \rho_2 \in \R$. Then 
    \begin{enumerate}
        \item For any $s \in \R$ there exists $C_{a, s} >0$ such that
        \begin{equation}\label{C.V.manifold}
            \norm{\OpM(a) u }_{s -\rho_1} \leq C_{a,s} \norm{u}_s;
        \end{equation}

        \item $\OpM(a)^*=\OpM(\overline{a})$ (here $^*$ denotes the adjoint operator) and in particular purely imaginary  symbols correspond to skew adjoint operators; 

        \item There exists a symbol $c \in \bS^{\rho_1+\rho_2}(T^*M)$ such that $\OpM(c)=\OpM(a) \circ \OpM(b)\ \operatorname{mod}\Psi^{-\infty}(M)$. Moreover, the operator $\OpM(c)$ is given by 
        $$
        \OpM(c)=\OpM\left(ab +\frac{1}{2i} \{ a, b\} \right)\ \operatorname{mod}\ \Psi^{\rho_1+ \rho_2-2}(M);
        $$

        \item The symbol for the commutators between $\OpM(a)$ and $\OpM(b)$ is given by
        \begin{equation}\label{p.bracket}
            [\OpM(a), \OpM(b)]= -i \OpM(\{a, b\})\ \operatorname{mod} \Psi^{\rho_1+\rho_2-2}(M); 
        \end{equation}
        \item if $\rho_1<0$, then $\OpM(a)$ is a compact operator.
    \end{enumerate}
\end{lemma}
Observe that item 2 follows from our choice of isochore charts. Finally, we also need the following G\aa{}rding inequality:
\begin{lemma}\label{Garding}
    Let $a \in \bS^\rho(T^*M)$, $a \geq 0$. Then, for every $u \in C^\infty(M)$, we have
    \begin{equation}
        \la \OpM(a)u, u \ra_{L^2(M)}\geq - C_{\rho, a}\norm{u}^2_{\cH^{\frac{\rho-1}{2}}(M)}.
    \end{equation}
\end{lemma}
We refer to \cite[Theorem 9.11]{Zworski} for a proof in $\R^n$ from which the case on $M$ follows. 

\subsection{Generalized transport equations}\label{ss.general.transport}
In the rest of this section, we prove our main result, namely Theorem \ref{main.theorem}. The starting point is to rewrite the transport equation \eqref{transport.eq} in terms of the symbols associated to pseudodifferential operators in view of using the above microlocal tools. Precisely, using Lemma \ref{order.one.symbol}, we can in fact consider more generally the transport like equation: 
\begin{equation}\label{pseudo.transport}
    \pa_t u_\e = \OpM(i \xi(V(x) + \e \cP_\e(t, x))) u_\e + \OpM(ib_\e(t,x, \xi))u_\e, \quad x \in M, \ t \in \R, \ \e >0, 
\end{equation}
where 
\begin{itemize}
\item $V$ belongs to $\fX^{MS}(M)$,
\item $\mathcal{P}_\e(t,x)$ is a smooth time dependent vector field belonging to $\mathcal{C}^0_b(\R; \fX(M))$ with all semi-norms uniformly bounded in terms of $0\leq\e\leq 1$,
\item there exists $\rho<1$ such that $b_{\e,1}(t,x,\xi):=\text{Re}\,b_\e(t,x,\xi) \in C^0_b(\R; \bS^\rho(T^*M))$ with all seminorms uniformly bounded in terms of $0\leq\e\leq 1$,
\item $b_{\e,2}(t,x):=\e^{-1}\text{Im}\,b_\e(t,x) \in C^0_b(\R; \bS^0(T^*M))$ with all seminorms uniformly bounded in terms of $0\leq\e\leq 1$.
\end{itemize}
Note that the function $b_\e(t,x,\xi)$ appearing in Theorem~\ref{main.theorem} would not depend on $\xi$ (thanks to Lemma \ref{order.one.symbol}) and that it would be of order $0$ (and proportional to $\e$). Yet, as our proof allows to handle this more general case, we consider such a general $b_\e$.

\begin{remark}According to~\cite[Th.~1.2]{MasperoRobert}, one has that, for any $\sigma\geq 0$ and for any $u_0\in\mathcal{H}^\sigma(M)$, the solution $u(t,x)$ exists globally in time and belongs to $\mathcal{C}^0(\R; \mathcal{H}^\sigma(M))$ when $b_{\e,2}\equiv 0$. Here, we want to allow this extra selfadjoint term and we postpone the existence of the solution to Appendix~\ref{a:existence} following the results from this reference and classical perturbative arguments from ordinary differential equations.
\end{remark}

Our main result reads 
\begin{theorem}\label{main.theorem.pseudo} Suppose that the above assumptions on $V$, $\mathcal{P}_\e$ and $b_\e$ hold and, if $n=1$, suppose in addition that $V$ has at least one critical point.

There exists $\delta_{0}>0$ such that, for every $0\leq \sigma\leq \delta_0$ and for every $0\leq \e \leq \delta_0\sigma^2$, one can find $v_\e\in\mathcal{H}^\sigma(M)$ and $r:=r(\e,\sigma,v_\e)>0$ such that 
\begin{equation}\label{growth.pseudo}
\|u_\e(0)-v_\e\|_{\sigma}\leq r\quad \implies\quad   \norm{u_\e(t)}_\sigma \geq \delta_0e^{\delta_0\sigma t}\|u_\e(0)\|_\sigma,\quad \text{for any}\ t>0,
\end{equation} 
where $u_\e(t)$ is the solution to equation \eqref{pseudo.transport} with initial datum $u_\e(0).$
\end{theorem}

Observe that this Theorem implies Theorem \ref{main.theorem}. Yet the equation considered in~\eqref{pseudo.transport} is slightly more general and it encompasses the case where the right-hand side of the equation is not necessarily skew adjoint. We emphasize that the selfadjoint perturbation in \eqref{pseudo.transport} is small both in terms of $\e$ and of the order. As anticipated, the key ingredient of our proof is the existence of an escape function for $h(x, \xi)=\xi (V(x))$, proven in Proposition \ref{prop.escape}. Thus, we start associating a symbol to the escape function.

Let $f$ be any positively homogeneous function of order $\sigma>0$, in the sense of Definition \ref{def.pos.homo}. Let $\chi \in C_c^\infty(T^*M; [0,1])$ be a bump function on $T^*M$, such that 
    \begin{equation}\label{shape.chi}
        \mbox{supp}(\chi) \subseteq \{ (x, \xi) : x \in M, \norm{\xi}_x \leq 2 \}, \quad \chi(x, \xi) \equiv 1, \quad \forall \norm{\xi}_x \leq 1. 
    \end{equation}
    Then the function $\tilde{f}(x, \xi):=f(x,\xi)(1-\chi(x,\xi))$ is a symbol of order $\sigma$, i.e., $\tilde{f} \in \bS^\sigma(T^*M)$. Indeed, it is immediate to see from its definition that $\tilde{f}$ is a smooth function on the whole $T^*M$. Moreover using the characterization in \eqref{set.symbols.M} of $\bS^\sigma(T^*M)$, one obtains the decay property of symbols from homogeneity of $f$. 
    
    In particular, for every choice of $0<\sigma<\sigma_0$ (where the threshold $\sigma_0$ is given by Proposition \ref{prop.escape}), the escape function $a$ for $h(x, \xi)=\xi (V(x))$ is a positively homogeneous function of order $\sigma$, where $\sigma$ depends on the vector field $V$ (recall the Definition \ref{def.escape} of escape function). Thus we denote by 
    \begin{equation}\label{smooth.escape}
        \tilde{a}(x, \xi):= a(x, \xi) ( 1- \chi(x, \xi)) \in \bS^\sigma(T^*M), 
    \end{equation}
    a smooth symbol associated with the escape function. We remark that, choosing $\chi$ as in \eqref{shape.chi}, we have  $\tilde{a} \equiv a$ in $\{ (x, \xi) \in T^*M : \norm{\xi}_x \geq 2\}$, so that, in this set, conditions \eqref{bound.escape} and \eqref{bound.escape.2} hold for $\tilde{a}$ as well.

\subsection{Energy type estimate}\label{ss.energy.estimate}

Let $u_\e(t, x)$ be a solution of equation \eqref{pseudo.transport} and define 
\begin{equation}\label{At}
    \cA(t):= \la \OpM(- \tilde{a})u_\e(t, x), u_\e(t, x) \ra_{L^2(M)}, \qquad \forall t \in \R,
\end{equation}
where $\tilde{a}$ is given in \eqref{smooth.escape}. The main result of this paragraph is the following Lemma.
\begin{lemma}\label{lemma.growth.A}
    There exist $\e_0>0$, $0<\sigma_0<1$ and $\tilde\alpha>0$ such that, for any $0<\sigma\leq\sigma_0$, one can find $\beta_\sigma>0$ so that, for every $0<\e\leq\sigma^2\e_0$, the function $\cA(t)$ in \eqref{At} satisfies
\begin{equation}\label{growth.A}
    \frac{\di}{\di t} \cA(t) \geq \tilde\alpha \sigma \cA(t)- \beta_\sigma\|u_\e(t)\|_{L^2}^2,\quad\forall \ t \in \R, 
\end{equation}
where $u_\e(t, x)$ is the solution of \eqref{transport.eq} with initial condition $u_\e(0, x)$.
\end{lemma}

For simplicity of notations, from now on we drop the index $\e$ and just write $u_\e=u$. Before proving the Lemma, we remark that, from the properties of the operator $\mathcal{B}(t):=\OpM(i\xi(V(x)+\e\mathcal{P}(t,x))+b_\e(t,x,\xi))$ and the definitions of the functions $b_{\e_1}$, $b_{\e_2}$ in \eqref{pseudo.transport}, we have 
\begin{equation}\label{e.normL2}
\frac{\di}{\di t}\|u(t)\|_{L^2}^2=2\e \langle\OpM(b_{\e,2}) u(t),u(t)\rangle\leq 2C_b\e \|u(t)\|_{L^2}^2,
\end{equation}
where we used the Calder\'on-Vaillancourt property~\eqref{e.cv} in the last inequality and where the involved constant $C_b>0$ is uniform for $0\leq\e\leq 1$. In particular,
\begin{equation}\label{e.bound.normL2}
\|u(t)\|_{L^2}\leq e^{C_b\e t} \|u(0)\|_{L^2}, \quad \forall t\geq 0, 
\end{equation}
and equation~\eqref{growth.A} also reads, for any $c>0$,
\begin{equation}\label{growth.A.bis0}
    \frac{\di}{\di t} \left(\cA(t)- c\|u(t)\|_{L^2}^2\right) \geq \tilde\alpha\sigma \left(\cA(t)- \frac{\beta_\sigma+2cC_b\e}{\tilde\alpha\sigma}\|u(t)\|_{L^2}^2\right),\quad\forall \ t \in \R. 
\end{equation}
Hence, if we set $
c=\beta_{\sigma,\e}:=\frac{\beta_\sigma}{\tilde\alpha\sigma-2C_b\e}$ for $0\leq\e \leq \tilde\alpha \sigma/4C_b$ (recall that $0\leq \e \leq\e_0\sigma^2$),
we find that
\begin{equation}\label{growth.A.bis}
    \frac{\di}{\di t} \left(\cA(t)- \beta_{\sigma,\e}\|u(t)\|_{L^2}^2\right) \geq \tilde\alpha\sigma \left(\cA(t)- \beta_{\sigma,\e}\|u(t)\|_{L^2}^2\right),\quad\forall \ t \in \R. 
\end{equation}
In paragraph~\ref{ss.initial.data}, we will explain how one can derive Theorem~\ref{main.theorem.pseudo} from this energy estimate in \eqref{growth.A.bis}.

\subsection{Proof of Lemma \ref{lemma.growth.A}}
First of all, we remark that, using \eqref{pseudo.transport} and the properties of the operator $\mathcal{B}(t)$, we can write: 
\begin{equation}\label{der.A.0}
\begin{split}
    \frac{\di}{\di t} \cA(t) 
    & = \la [\OpM(-\tilde{a}), \OpM(i\xi (V(x)+ \e \cP_\e(t, x))+ib_{\e,1}(t,x, \xi)) ] u, u \ra
    \\
    & +2\e \text{Re} \la \OpM(\tilde{a})\OpM(b_{\e,2})u, u \ra .
    \end{split}
\end{equation}
Using the G\aa{}rding inequality from Lemma~\ref{Garding}, one can find a constant $C_{b}$ (uniform for $t\in\R$ and $0\leq\e\leq 1$) such that this equality implies
\begin{equation}\label{der.A}
    \frac{\di}{\di t} \cA(t) 
     \geq  \la [\OpM(-\tilde{a}), \OpM(i\xi (V(x)+ \e \cP_\e(t, x))+ib_{\e,1}(t,x, \xi)) ] u, u \ra -C_b \e \la \OpM(\langle\xi\rangle^\sigma)u, u \ra ,
\end{equation}
where $\langle\xi\rangle^\sigma:=(1+\|\xi\|_x^2)^{\frac{\sigma}{2}}.$

Hence, we are left with studying the operator $[\OpM(-\tilde{a}), \OpM(i\xi (V(x)+ \e \cP_\e(t, x))+ib_{\e,1})]$. We first use Lemma \ref{manifold.symbolic} to recover its symbol: 
\begin{multline}\label{symbol.c}
    [\OpM(-\tilde{a}), \OpM(i\xi (V(x)+ \e \cP_\e(t, x))+ib_{\e_1})] \\
    \stackrel{\eqref{p.bracket}}{=} \OpM(-i \{ -\tilde{a}, i \xi (V(x) + \e \cP_\e(t))+ib_{\e,1}\})\ \mod\ \Psi^{\sigma-1}(M) \\
    = \OpM(\{ \xi (V(x)+ \e \cP_\e(t))+b_{\e,1}, \tilde{a} \} )\ \mod\ \Psi^{\sigma-1}(M)\\
     = \OpM(\{ \xi (V(x)), \tilde{a} \}) + \OpM(\e\cP_{\tilde{a}}(t)+\cP_{b,\e}(t))\ \mod\ \Psi^{\sigma-1}(M)
\end{multline}
where the remainder in $\Psi^{\sigma-1}(M)$ results from the symbolic calculus operations (see \eqref{p.bracket}) with all the seminorms of the symbols uniformly bounded in terms of $t\in\R$, $0<\sigma\leq 1$ and $0<\e <1$. Here, we have denoted by $\cP_{\tilde{a}}(t):= \{\xi (\cP_\e(t,x)),\tilde{a} \}$ and $\cP_{b,\e}(t):= \{b_{\e,1},\tilde{a} \}$, just to shorten the notation. Remark that $\cP_{\tilde{a}}(t) \in C^0_b(\R; \bS^\sigma(M))$, that it depends implicitly on $\e$ and that it is a homogeneous function of degree $\sigma$ on $\{(x, \xi) : \norm{\xi}_x \geq 2 \}$ for all $t \in \R$. Similarly, $\cP_{b,\e}(t)$ belongs to $\mathcal{C}^0(\R; \bS^{\sigma+\rho-1}(M))$. Consider now a cutoff function $\chi_3 \in C_c^\infty(T^*M; [0, 1])$, such that
\begin{equation}\label{chi.3}
    \mbox{supp}(\chi_3) \subseteq \{ (x, \xi) : \norm{\xi}_x \leq R_0+1 \}, \quad \mbox{ and } \quad \chi_3(x, \xi) \equiv 1 \mbox{ in } \{ (x, \xi) : \norm{\xi}_x \leq R_0\},
\end{equation}  
where $R_0\geq 2$ will be determined later on. We can rewrite 
\begin{multline}\label{split.cutoff}
     \{\xi (V(x)), \tilde{a} \}+ \e \cP_{\tilde{a}}(t)+\mathcal{P}_{b,\e}(t)-C_b\e\langle\xi\rangle^\sigma= (\{\xi (V(x)), \tilde{a} \}+ \e \cP_{\tilde{a}}(t)+\mathcal{P}_{b,\e}(t)-C_b\e\langle\xi\rangle^\sigma)(1-\chi_3)   \\
     + (\{\xi (V(x)), \tilde{a} \}+ \e \cP_{\tilde{a}}(t)+\mathcal{P}_{b,\e}(t)-C_b\e\langle\xi\rangle^\sigma )\chi_3.
\end{multline}
Since the cutoff function is compactly supported, we have from Lemma \ref{symbols.compact} that
\begin{equation}\label{compact}
    \OpM((\{\xi (V(x)), \tilde{a} \}+ \e \cP_{\tilde{a}}(t)+\mathcal{P}_{b,\e}(t)-C_b\e\langle\xi\rangle^\sigma)\chi_3) \in C^0_b(\R; \cS^{-\nu}), \quad \forall \nu>0,
\end{equation}
with the involved constants in the semi-norms being uniform for $t\in\R$, $0\leq \e\leq 1$, $0\leq\sigma\leq 1$ and $R_0$ varying in a compact interval of $(0,\infty)$. Thus the right-hand side of \eqref{der.A} reads:  
\begin{multline}
\label{comm.with.R}
    [\OpM(-\tilde{a}), \OpM(i\xi (V(x)+ \e \cP_\e(t, x))+ib_{\e_1})]-C_b\e\OpM(\langle\xi\rangle^\sigma)\\
    \stackrel{\substack{\eqref{split.cutoff}\\ \eqref{compact}}}{=} \OpM((\{\xi (V(x)), \tilde{a} \}+ \e \cP_{\tilde{a}}(t)+\mathcal{P}_{b,\e}(t)-C_b\e\langle\xi\rangle^\sigma)(1-\chi_3) ) \ \mod\ \Psi^{\sigma-1}(M),
\end{multline}
where the remainder takes into account both the remainder in \eqref{symbol.c} and the compact operator in \eqref{compact}. One more time, all the seminorms of the symbols in this remainder are uniformly bounded in time and in terms of $0<\e<1$ and $0<\sigma\leq 1$. Yet it is worth noting that the constants depend on $R_0\geq 2$.

We now claim that, from \eqref{bound.escape} and \eqref{bound.escape.2}, there exist $\e_0,\sigma_0>0$ small enough and a positive constant $C_+$ such that, for every $0< \sigma\leq\sigma_0$, one can find $R_0\geq 2$ so that, for every $0\leq \e \leq \e_0\sigma^2$,
\begin{equation}\label{claim.final}
    (\{ \xi V(x), \tilde{a} \} + \e \cP_{\tilde{a}}(t)+\mathcal{P}_{b,\e}(t)-C_b\e\langle\xi\rangle^\sigma)(1-\chi_3) > -C_+\sigma \tilde{a}(1-\chi_3), \quad \forall (x, \xi) \in T^*M, \forall t \in \R. 
\end{equation}
Let us postpone the proof of \eqref{claim.final} and conclude the proof of \eqref{growth.A}. Applying G\aa{}rding inequality (see Lemma \ref{Garding}) with $0<\sigma<1$, we indeed obtain from \eqref{claim.final} that, for every $0< \e < \e_0\sigma^2$: 
\begin{equation}\label{use.garding}
    \OpM((\{\xi V(x),  \tilde{a}\} + \e \cP_{\tilde{a}}(t)+\mathcal{P}_{b,\e}(t)-C_b\e\langle\xi\rangle^\sigma)(1-\chi_3)) \geq 
    C_+\sigma \OpM(-\tilde{a}) -C_{0}\text{Id},
\end{equation}
where $C_{0}>0$ is the constant resulting from G\aa{}rding inequality and the remainder in $\Psi^{\sigma-1}(M)$. It is again uniformly bounded for $t\in\R$ but now depends on $0<\sigma\leq 1$ (through its dependence in $R_0$).
Plugging \eqref{comm.with.R} and \eqref{use.garding} in \eqref{der.A} we obtain 
\begin{equation}\label{done!}
    \frac{\di}{\di t} \cA(t)
    \geq C_+\sigma \la \OpM(-\tilde{a})u, u \ra - (C_{K}+C_0)\|u(t)\|_{L^2}^2,
\end{equation}
with $C_{K}>0$ taking into account the compact remainder from~\eqref{compact}. Observe that again this remainder depends implicitly on $R_0\geq 2$ and thus on $0\leq \sigma\leq\sigma_0$. See Remark~\ref{remark.R0} for discussion on this dependence.
This concludes the proof, since \eqref{done!} is \eqref{growth.A} with $\tilde{\alpha}:=C_+$ and $\tilde{\beta}_\sigma:=C_{K}+ C_0>0$.

We are left to prove \eqref{claim.final}. To this aim, we first recall that, from the properties of $\chi_3$ in \eqref{chi.3}, we only have to verify \eqref{claim.final} in $\{(x, \xi) \in T^*M : \norm{\xi}_x \geq R_0 \}$, since $(1-\chi_3)\equiv 0$ otherwise. We thus claim that 
\begin{equation}\label{too.much}
    \{ \xi V(x), \tilde{a} \} + \e \cP_{\tilde{a}}(t)+\mathcal{P}_{b,\e}(t)-C_b\e\langle\xi\rangle^\sigma> - C_+\sigma \tilde{a}, \quad \mbox{ in } \{ (x, \xi) \in T^*M : \norm{\xi}_x \geq R_0\},
\end{equation}
for some $C_+>0$ (depending only $V$) and $0<\e<\e_0\sigma^2$ with $\e_0$ to be determined. Clearly \eqref{too.much} implies \eqref{claim.final} since $(1-\chi_3)\geq0$ everywhere. We remark that in this remaining region $\tilde{a} \equiv a$, where $a$ is the homogeneous escape function given by Proposition \ref{prop.escape}. To prove \eqref{too.much} we split the domain as 
\begin{equation}\label{two.sets}
    \{(x, \xi) \in T^*M : \norm{\xi}_x \geq R_0 \} = \underbrace{\{(x, \xi) \in E : \norm{\xi}_x \geq R_0 \}}_{=:\cE} \cup \underbrace{\{(x, \xi) \in T^*M\setminus E : \norm{\xi}_x \geq R_0 \}}_{=: \cI},
\end{equation}
where $E$ is the set appearing in the Definition \ref{def.escape} of the escape function. First, we look at the set $\cI$. Using \eqref{bound.escape} (with $\delta=c_0 \sigma$ given by Proposition \ref{prop.escape}), we have that 
\begin{equation}\label{In.I}
\{ \xi V(x), \tilde{a} \} + \e \cP_{\tilde{a}}(t) +\cP_{b,\e}(t)-C_b\e\langle\xi\rangle^\sigma\geq (c_0\sigma - \e (M_2+2C_b)-\tilde{C}_bR_0^{\rho-1})\norm{\xi}_x^\sigma,
\end{equation}
where $\tilde{C}_b>0$ is a constant (uniform in $0\leq \e\leq 1$ and $0\leq\sigma\leq 1$) that depends on the seminorms of $b_\e$ and where $M_2 := \max_{t \in \R, (x, \xi) \in \cI} \left|\cP_{\tilde{a}}\left(t, x, \frac{\xi}{\norm{\xi}}\right) \right|>0$ is again uniform for $0<\sigma\leq 1$. Moreover, using again homogeneity of $a$, one has $a(x, \xi) \geq - C_a \norm{\xi}^\sigma_x$ for some $C_a>0$ (that is uniform for $0\leq\sigma\leq 1$) and for all $(x, \xi) \in \cI$. Thus 
choosing $\e_{\cI}:=\frac{c_0}{4(M_2+2 C_b)}>0$ and $R_0^{\rho-1}\leq \frac{c_0\sigma}{4\tilde{C}_b}$, we obtain from \eqref{In.I} that, for all $0<\e<\e_{\cI}\sigma$ we have: 
\[
\{\xi (V(x)), \tilde{a} \} + \e \cP_{\tilde{a}}(t)+\cP_{b,\e}(t) -C_b\e\langle\xi\rangle^\sigma\geq \frac{c_0\sigma}{4} \norm{\xi}_x^\sigma \geq - \frac{c_0}{4C_a}\sigma a.
\]
Finally, from condition \eqref{bound.escape.2} (with $\delta= c_0 \sigma$ from Proposition \ref{prop.escape}), we see that, in $\cE$ (see \eqref{two.sets}):
\begin{equation}\label{In.E}
\tilde{a} \equiv a \geq c_0\sigma \norm{\xi}^\sigma_x  \quad \mbox{whilst} \quad \{ \xi V(x) , a \}  + \e \cP_{\tilde{a}}(t) +\mathcal{P}_{b,\e}(t)-C_b\e\langle\xi\rangle^\sigma\stackrel{\eqref{bound.escape.2}}{\geq}
- (\e  (M_1+C_b)+\tilde{C}_bR_0^{\rho-1}) \norm{\xi}^\sigma_x, 
\end{equation}
where $M_1:=\max_{t \in \R, (x, \xi) \in \cE}\left|\cP_{\tilde{a}}\left(t, x, \frac{\xi}{\norm{\xi}_x}\right) \right|\geq0$ (since $\cP_{\tilde{a}}(t)$ is positively homogeneous of degree $\sigma$ in $\cE$) is uniformly bounded in terms of $0<\sigma\leq 1$. Thus choosing $\e_{\cE}:= \frac{c_0^2}{8(M_1+C_b)C_a}$ and  $R_0^{\rho-1}\leq \frac{c_0^2\sigma^2}{8\tilde{C}_bC_a}$, we obtain that, for all $0<\e<\e_{\cE}\sigma^2$,
$$
\{ \xi V(x) , a \}  + \e \cP_{\tilde{a}}(t) +\mathcal{P}_{b,\e}(t)- C_b \e \la \xi \ra^\sigma \geq 
- \frac{c_0^2\sigma^2}{4C_a} \norm{\xi}^\sigma_x\geq -\frac{c_0\sigma}{4C_a}a. 
$$
This proves \eqref{too.much} with $C_+:=\frac{c_0}{4C_a}>0$ and $\e_0:= \min\{\e_{\cI}, \e_{\cE}\}>0$, concluding the proof of the energy estimate~\eqref{growth.A}.

\begin{remark}\label{remark.R0} Observe that, the smaller $\sigma$ is, the larger $R_0\geq 2$ has to be. If $b_{\e,1}$ was of the form $\e \tilde{b}_{\e,1}$ with $\tilde{b}_{\e,1}\in\mathcal{C}^0_b(\R; \bS^\rho(M))$ having all its seminorms uniformly bounded in terms of $0\leq \e\leq 1$, one could in fact pick $R_0=2$ up to shrinking the value of $\e_0$. In particular, $\tilde{\beta}_\sigma$ could be chosen uniformly in terms of $0\leq\sigma\leq \sigma_0$. Finally, note that the choice of the parameter $\e_0$ depends not only on $V$ and $\mathcal{P}_\e$ but also on $b_{\e,2}$. 
\end{remark}

\subsection{Initial data and end of the proof of Theorem~\ref{main.theorem.pseudo}}\label{ss.initial.data}

We are now ready to prove Theorem~\ref{main.theorem.pseudo}. According to Equation~\eqref{growth.A.bis}, one has that, for any $0<\sigma\leq\sigma_0$, for any $0<\e\leq \sigma^2 \e_0$ and for any initial datum $u_0\in L^2$,
\begin{equation}\label{integrate}
    \cA(t) \geq e^{\tilde\alpha\sigma  t}\left(\cA(0)-\beta_{\sigma,\e}\|u_0\|_{L^2}^2 \right) + \beta_{\sigma,\e}\|u(t)\|_{L^2}^2, \qquad \forall t \geq 0. 
\end{equation}
Cauchy-Schwarz inequality together with~\eqref{e.bound.normL2} now gives: 
\begin{equation}\label{finish}
\cA(t) \stackrel{\eqref{At}}{=} \la \OpM(-\tilde{a}) u(t, x),  u(t, x) \ra_{L^2(M)} \leq \norm{\OpM(-\tilde{a}) u}_{L^2(M)}e^{C_b\e t}\|u_0\|_{L^2}.
\end{equation}
Since $\tilde{a} \in \bS^\sigma(T^*M)$, Lemma \ref{manifold.symbolic} (inequality \eqref{C.V.manifold}) gives:
\begin{equation}\label{eq.2}
\norm{\OpM(\tilde{a})u}_{L^2(M)}\leq C_{\tilde{a}} \norm{u}_{\cH^\sigma}, \quad \forall u \in C^\infty(M). 
\end{equation}
Thus, plugging \eqref{eq.2} and \eqref{finish} in~\eqref{integrate}, we have, for every $0<\sigma\leq\sigma_0<1$, 
\begin{equation}\label{e.exp.lower.bound}
C_{\tilde{a}}\norm{u}_{\cH^\sigma}\|u_0\|_{L^2} \geq e^{(\tilde\alpha\sigma-C_b\e) t}\left(\cA(0)-\beta_{\sigma,\e}\|u_0\|_{L^2}^2 \right)\geq e^{(\tilde\alpha-C_b\e_0\sigma ) \sigma t}\left(\cA(0)-\beta_{\sigma,\e}\|u_0\|_{L^2}^2 \right). 
\end{equation}
Hence, we will be done with the proof of Theorem~\ref{main.theorem.pseudo} if we are able to find initial datum $u_0$ such that 
\begin{equation}\label{e.criterium.growth}
 F(u_0):=-(\cA(0)-\beta_{\sigma, \e}\norm{u_0}^2_{L^2}) \equiv \la \OpM(\tilde{a})u_0, u_0 \ra_{L^2(M)}+\beta_{\sigma,\e}\|u_0\|_{L^2}^2<0.
\end{equation}

\begin{remark}
  Note that the proof is valid for any $0<\sigma<\sigma_0$ with $\sigma_0$ given by Proposition \ref{prop.escape}. Yet, since for every $s \geq \sigma$, we have 
$\norm{u}_s \geq \norm{u}_\sigma$, exponential blow up of the Sobolev norm of the solutions to~\eqref{pseudo.transport} holds for any $s \geq \sigma$. 
\end{remark}

Let us now construct $u_0$ verifying~\eqref{e.criterium.growth}. To that aim, we recall property~\eqref{negative.escape} of the escape function $a$, namely that it is $\leq -c_0\|\xi\|^\sigma/2$ in some non empty\footnote{This is exactly for this reason that we need to impose that $V$ has a critical point when $n=1$.} closed conic subset $C$. Hence, we will construct initial data which are microlocalized in this region where $a<0$. To that aim, we fix $\chi\in\mathcal{C}^\infty(M; [0,1])$ such that the support of $\chi$ intersects $C$ and such that the support of $\chi$ is contained in an isochore local chart $\gamma: \mathcal{U}_j\subset M\rightarrow B(0,1):=\{x\in\R^n:\|x\|<1\}$. In order to work with the Weyl quantization on $\R^n$, we will consider initial data of the form $u_0=\chi v_0\circ\gamma$ with $v_0\in\mathcal{C}^{\infty}_c(B(0,1))$. This allows us to write
\begin{equation}
\la \OpM(\tilde{a})u_0, u_0 \ra_{L^2(M)}=\la \gamma^{-1*}\chi\OpM(\tilde{a})\chi\gamma^* v_0, v_0 \ra_{L^2(\R^n)}.
\end{equation}
Thanks to the rule for the change of coordinates for pseudodifferential operators~\cite[Th.~9.9]{Zworski}, one finds that
\begin{equation}\label{bound.chi.a}
\la \OpM(\tilde{a})u_0, u_0 \ra_{L^2(M)}\leq \la \Opw((\chi^2\tilde{a})\circ\tilde{\gamma} ) v_0, v_0 \ra_{L^2(\R^n)} +C\|v_0\|_{L^2(\R^n)}^2,
\end{equation}
for some constant $C$ that depends only on the coordinate chart, on $\tilde{a}$ and on the choice of the quantization procedure $\OpM$ and with $\tilde{\gamma}$ defined in \eqref{symp.lift}. We now fix some $\xi_0\neq 0$ such that $\xi_0$ lies in the region where $(\chi^2\tilde{a})\circ\tilde{\gamma}\leq -c_\gamma\|\xi\|^{\sigma}$ and we set 
$$
\tilde{v}_h(x)=\chi_0(x)e^{\frac{i\xi_0\cdot x}{h}},\quad 0<h\leq 1,
$$ 
with $\chi_0\in\mathcal{C}^{\infty}_c(B(0,1))$ which is not identically $0$. With this definition at hand, one has
\begin{lemma}\label{l.norm.initial.data} There exists a constant $C_0>0$ (depending on $\xi_0$ and $\chi_0$) such that, for every $0<h\leq 1$ and for every $0\leq\sigma\leq 1$, one has
   \begin{equation}\label{e.sobolev.norm.estimate}
  C_0^{-1}\leq \|\tilde{v}_h\circ\gamma\|_{L^2}\leq C_0,\ \text{and}\  \|\tilde{v}_h\circ\gamma\|_{\sigma}\leq C_0h^{-\sigma}.
\end{equation}
\end{lemma}

\begin{proof} Regarding the $L^2$-norm, one has
$$
\|\tilde{v}_h\|_{L^2(\R^n)}=\int_{\R^n}|\chi_0(x)|^2dx,
$$
from which we deduce the expected equality as we have chosen an isochore chart. Regarding the Sobolev norm, one has
$$
\|\tilde{v}_h\|_\sigma^2=\int_{\R^n}(1+\|\xi\|^2)^\sigma\left|\int_{\R^n}\chi_0(x)e^{ix\cdot\left(\frac{\xi_0}{h}-\xi\right)}dx\right|^2d\xi,
$$
where $\|.\|_\sigma$ denotes here the Sobolev norm in $\R^n$ (endowed with the standard Euclidean metric). Hence, denoting by $\widehat{\chi}_0$ the Fourier transform of $\chi_0$, one finds
$$
\|\tilde{v}_h\|_\sigma^2=\int_{\R^n}\left(1+\left\|\xi+\frac{\xi_0}{h}\right\|^2\right)^\sigma|\widehat{\chi}_0(\xi)|^2d\xi\leq 2^\sigma\left(1+2\frac{\|\xi_0\|^2}{h^2}\right)^\sigma\|\chi_0\|_{\sigma}.
$$
The conclusion follows then from the fact that the Sobolev norm induced by the Riemannian metric $g$ on $M$ and the one induced by the Euclidean metric in the local chart are uniformly equivalent (with constants that can be made uniform in terms of $0\leq\sigma\leq 1$).
\end{proof}

\begin{remark}
Recall that $\tilde{a}$ is $\sigma$ homogeneous for $\|\xi\|_x\geq 2$ with $0<\sigma<1$. Hence, its pullback to $T^*\R^n$ is also $\sigma$-homogeneous for $\|\xi\|$ large enough, say $R_0$ (where $\|.\|$ now denotes the Euclidean norm in $\R^n$). Hence, we pick $\|\xi_0\|>2R_0$ in order to have this covector in the region where the pullback symbol $(\chi^2\tilde{a})\circ\tilde{\gamma}$ is $\sigma$-homogeneous.
\end{remark}
We now study the behavior of 
\begin{equation}\label{oscillating.integral}
\la \Opw((\chi^2\tilde{a})\circ\tilde{\gamma} ) \tilde{v}_h, \tilde{v}_h \ra_{L^2(\R^n)},\ \text{as}\ h\to 0^+.
\end{equation}
Observing that the $L^2$ norm of $\tilde{v}_h$ is uniformly bounded in terms of $0<h\leq 1$, we will be done with the proof of~\eqref{e.criterium.growth} if we can prove that~\eqref{oscillating.integral} tends to $- \infty$ as $h\to 0^+$.  Indeed, from Cauchy-Schwartz inequality and \eqref{bound.chi.a}, we will obtain \eqref{e.criterium.growth}. This is the content of the next Lemma.
\begin{lemma}\label{l.check.criterium} With the above conventions, there exist $0<c_1,h_1<1$ such that, for every $0<h<h_1$ and for every $0\leq \sigma \leq 1$,
$$
\la \Opw((\chi^2\tilde{a})\circ\tilde{\gamma} ) \tilde{v}_h, \tilde{v}_h \ra_{L^2(\R^n)}\leq -c_1h^{-\sigma}\leq  -c_1 C_0^{-2}\|\tilde{v}_h\|_{\sigma}\|\tilde{v}_h\|_{L^2}.
$$
\end{lemma}
Note that the last inequality follows from~\eqref{e.sobolev.norm.estimate} (thus the constant $C_0$ depends on $\xi_0$ and $\chi_0$). The above discussion combined with this Lemma yields that, for $h>0$ small enough (depending on $0<\sigma\leq 1$),
\begin{equation}\label{e.lower.bound.F}
F\left(\tilde{v}_h\circ\gamma \right)\leq  -\frac{c_1 C_0^{-2}}{2}\|\tilde{v}_h\circ\gamma \|_{\sigma}\|\tilde{v}_h\circ\gamma \|_{L^2}.
\end{equation}
As already explained, this yields the proof of Theorem~\ref{main.theorem.pseudo} with $v_\e=\tilde{v}_h\circ\gamma$. More precisely, we have shown the existence of a constant $\delta_1>0$ such that, for every $0<\sigma\leq\delta_1$ and for every $0<\e\leq\delta_1\sigma$, one can find $v_\e(0)\in\mathcal{H}^\sigma\setminus \{0\}$ (even $\mathcal{C}^\infty$) such that
\begin{equation}\label{e.center.ball.estimate}
\forall t>0,\quad\|v_\e(t)\|_\sigma\geq \delta_1\left(e^{\sigma\delta_1 t}\|v_\e(0)\|_\sigma+ \|v_\e(0)\|_{L^2}\right),
\end{equation}
where we have picked $\delta_1:=\min\{\tilde{\alpha},\tilde{\beta}/(C_{\tilde{a}}\tilde{\alpha}), c_1 C_0^{-2}/(2C_{\tilde{a}})\}$ and where $v_\e(t)$ is the solution to~\eqref{pseudo.transport} with initial datum $v_\e$.

We will explain below why this remains true in a small neighborhood of $v_\e(0)=\tilde{v}_h\circ\gamma$ and thus we will be done with the proof of Theorem~\ref{main.theorem.pseudo}.
\begin{proof}[Proof of Lemma~\ref{l.check.criterium}] We will make use of semiclassical pseudodifferential calculus as in~\cite{Zworski}. First, we introduce $\psi$ in $\mathcal{C}^{\infty}_c((-2,2);[0,1])$ which is identically equal to $1$ on $[-1,1]$ and we set $\psi_R(\xi):=\psi(\|\xi\|/R)$. Using non stationary phase Lemma~\cite{Zworski} and that $\|\xi_0\|>2R_0$, one can prove that $\|\Opw(\psi_{R_0h^{-1}})\tilde{v}_h\|_{L^2}=\mathcal{O}(h).$ Now using Calderon-Vaillancourt property~\ref{e.cv} with $s=\rho_1=\sigma$ and Lemma~\ref{l.norm.initial.data}, one finds that 
    $$
    \left\|\Opw((\chi^2\tilde{a})\circ\tilde{\gamma} ) \tilde{v}_h\right\|_{L^2}\leq \tilde{C}_a\left\|\tilde{v}_h\right\|_{\sigma}=\mathcal{O}(h^{-\sigma}).
    $$
Hence, gathering these two properties with the Cauchy-Schwarz inequality, one finds that
$$
\la \Opw((\chi^2\tilde{a})\circ\tilde{\gamma} ) \tilde{v}_h, \tilde{v}_h \ra_{L^2(\R^n)}=
\la \Opw(1-\psi_{R_0h^{-1}})\Opw((\chi^2\tilde{a})\circ\tilde{\gamma} ) \tilde{v}_h, \tilde{v}_h \ra_{L^2(\R^n)}+\mathcal{O}(h^{1-\sigma}).
$$
Using the composition rule for pseudodifferential operators together with~\eqref{e.cv} and the fact that $\|\tilde{v}_h\|_{L^2}=\mathcal{O}(1)$, we get
$$
\la \Opw((\chi^2\tilde{a})\circ\tilde{\gamma} ) \tilde{v}_h, \tilde{v}_h \ra_{L^2(\R^n)}=
\la \Opw((1-\psi_{R_0h^{-1}})(\chi^2\tilde{a})\circ\tilde{\gamma} ) \tilde{v}_h, \tilde{v}_h \ra_{L^2(\R^n)}+\mathcal{O}(1).
$$
Thanks to the introduction of the cutoff $(1-\psi_{R_0h^{-1}})$, the function $(\chi^2\tilde{a})\circ\tilde{\gamma} $ is now $\sigma$-homogeneous in the support of our cutoff function. Hence, using the semiclassical quantization $\Oph$ as in~\cite[Ch.4]{Zworski}, one finds
$$
\la \Opw((\chi^2\tilde{a})\circ\tilde{\gamma} ) \tilde{v}_h, \tilde{v}_h \ra_{L^2(\R^n)}=
h^{-\sigma}\la \Oph((1-\psi_{R_0})(\chi^2\tilde{a})\circ\tilde{\gamma} ) \tilde{v}_h, \tilde{v}_h \ra_{L^2(\R^n)}+\mathcal{O}(1).
$$
Let now $\tilde{\psi}\in\mathcal{C}^\infty_c(\R^n)$ which is equal to $1$ near $\xi_0$. Another application of the non stationary phase lemma yields that
$$
\la \Opw((\chi^2\tilde{a})\circ\tilde{\gamma} ) \tilde{v}_h, \tilde{v}_h \ra_{L^2(\R^n)}=
h^{-\sigma}\la \Oph(\tilde{\psi}(1-\psi_{R_0})(\chi^2\tilde{a})\circ\tilde{\gamma} ) \tilde{v}_h, \tilde{v}_h \ra_{L^2(\R^n)}+\mathcal{O}(1).
$$
The conclusion follows then from an application of the stationary phase Lemma (see~\cite[Ch.~5, Ex.~2]{Zworski} for the exact same integral) which gives
$$
\la \Oph(\tilde{\psi}(1-\psi_{R_0})(\chi^2\tilde{a})\circ\tilde{\gamma} ) \tilde{v}_h, \tilde{v}_h \ra_{L^2(\R^n)}=\int_{\R^d}(\chi^2\tilde{a})\circ\tilde{\gamma}(x,\xi_0) \chi_0^2(x)dx +o(1).
$$
By construction, this quantity is negative thanks to~\eqref{negative.escape} as $h\to 0^+$.
\end{proof} 

We can now conclude the proof of Theorem~\ref{main.theorem.pseudo}.
\begin{proof}[Proof of Theorem~\ref{main.theorem.pseudo}] Thanks to~\eqref{e.center.ball.estimate}, we have already the candidate for the center of the ball $v_\e(0)\neq 0$ in $\mathcal{H}^\sigma$ and it remains to prove that exponential growth is true in a small neighborhood of $v_\e(0)$. We let $\|u_\e(0)-v_\e(0)\|_\sigma\leq r$ with $r>0$ to be determined. Thanks to~\eqref{e.criterium.growth}, one has exponential growth if $F(u_\e(0))<0$. By continuity of the map $F$ on $\mathcal{H}^\sigma$ and thanks to~\eqref{e.lower.bound.F}, one has that, for $r>0$ small enough, 
$$
F(u_\e(0))\leq \frac{F(v_\e(0))}{2}\leq -\frac{c_1 C_0^{-2}}{4}\|v_\e(0)\|_\sigma\|v_\e(0)\|_{L^2}.
$$
Hence, using~\eqref{e.exp.lower.bound}, one finds that, for $t\geq 0$, 
$$
C_{\tilde{a}}\|u_\e(t)\|_\sigma\|u_\e(0)\|_{L^2}\geq e^{\tilde{\alpha}\sigma t} \frac{c_1 C_0^{-2}}{4}\|v_\e(0)\|_\sigma\|v_\e(0)\|_{L^2}.
$$
Choosing $r\leq \|v_\e(0)\|_{L^2}/2$, one has, for every $0\leq \sigma\leq 1$, 
$$
\|u_\e(0)\|_\sigma\leq r+\|v_\e(0)\|_\sigma\leq\frac{3}{2}\|v_\e(0)\|_\sigma,
$$
from which we infer the expected result
$$
C_{\tilde{a}}\|u_\e(t)\|_\sigma\geq e^{\tilde{\alpha}\sigma t} \frac{c_1 C_0^{-2}}{9}\|u_\e(0)\|_\sigma.
$$
\end{proof}
\begin{remark} Observe that our proof dealt with the general equation~\eqref{pseudo.transport} whose principal symbol is given by a (time-dependent) transport perturbation term. In fact, instead of $\e\xi(\mathcal{P}(t,x))$, we could also have considered perturbations of the form $\e p_\e(t,x,\xi)$ with $p_\e\in\mathcal{C}^0_b(\R;\bS^1(M)))$ with $p_\e$ real valued and with all the seminorms of $p_\e$ uniformly bounded in terms of $0\leq\e\leq 1.$
\end{remark}

\section{Application to periodic transport equations}\label{sec.2D}

In this last section we prove Theorem \ref{theorem.torus} and its extension in dimension $n>2$. Let us recall the equation we consider (see equation \eqref{transport.torus}):  
\begin{equation}\label{transp.torus}
    \pa_t u= (\nu + \e V(t,x)) \cdot \nabla_x u + \frac{\e}{2} \div(V(t,x)) u, \qquad x \in \T^n, \, t \in \R,
\end{equation}
where $\nu \in \N_+^n$, $V(t, x) \in C^\infty(\T \times \T^n)$ and $\div(\cdot)$ denotes the divergence with respect to the Euclidean volume on $\T^n$. In this section, we often drop the index $\e$ for the solution of equation \eqref{transp.torus} and we simply use $u(t, x)$ even if the solution depends on the parameter $\e>0$ (which requires some attention at certain steps).  

Let us start by giving a scheme of the proof, in order to explain why this result comes as an application of Theorem \ref{main.theorem}. We first perform a \emph{resonant normal form reduction} adapted from \cite{BambusiLangellaMontalto}, conjugating equation \eqref{transp.torus} to an equation with the same structure as \eqref{pseudo.transport}, but with leading transport term which is completely resonant with respect to the frequency $\nu$ (see Definition \eqref{def.res.average}). This reduces the study of instabilities of equation \eqref{transp.torus}, to that of an equation analogous to the one previously studied in this article. 
Indeed, after the normal form procedure, the transport term $\nu + \e V(t,x)$ in \eqref{transp.torus} is reduced to a time independent one (up to small-in-size reminders), given by the resonant average $$
\langle V\rangle_\nu(x):=\frac{1}{2\pi}\int_0^{2\pi}V(\tau, x-\tau \nu)d\tau\
\in\ C^\infty(\T^n)
$$ 
of $V$. 
Equivalently, $\la V \ra_\nu$ is obtained selecting only the resonant Fourier modes of $V$ (see \eqref{eq.res.average}) with respect to the periodic flow $\varphi^\tau(x)=x+\tau\nu$, $x\in\T^n$. Loosely speaking, these Fourier modes are the ones provoking the unstable dynamics.

Next, using Theorem \ref{main.theorem}, we can identify the set of perturbations $V$ such that their resonant average $\la V \ra_\nu$ will provoke unstable dynamics. 
Finally, we show that, for $n=2$, such set of potentials is generic. Let us remark that it is for this last part of the proof that we need to ask that $M=\T^2$. We are indeed able to perform the resonant normal form on all tori, i.e. $M=\T^n$, for all $n \in \N$, but in order to achieve genericity, we need to use the fact that Morse-Smale vector fields are generic in compact manifolds of dimension two (see Peixoto \cite{Peixoto1962} and Remark \ref{generic.2d}). 
This is the reason why the only possible choice (through our approach) to state Theorem \ref{theorem.torus}
is to pick $M=\T^2$ if we aim at proving genericity among smooth perturbations.

\subsection{The Resonant Normal Form}\label{section.normalform}
In this section, we perform the normal form reduction. We use the ideas from \cite{BambusiLangellaMontalto} and we adapt them to our context. First of all, given a $C^\infty$ function $V : \T \times \T^n \to \R^n$, with components $V(t, x):=(V_1(t,x),\ldots, V_n(t,x)),$ 
we use the standard compact notation for the Fourier coefficients of $V$: 
\begin{equation}\label{fourier}
V(t,x):= \sum_{\substack{k \in \Z^n\\ \ell \in \Z}} \underbrace{v_{k, \ell}}_{\in \C^n} e^{i(k\cdot x + \ell t)}, \qquad  \forall (t, x) \in \T \times \T^n,
\end{equation}
where $\{v_{k, \ell}\}_{k, \ell}$ are the Fourier coefficients of $V$. Then, we can define the resonant average of a function as follows.
\begin{definition}[Resonant average]\label{def.res.average}
    Let $V(t, x) \in C^{\infty}(\T^{n+1}; \R^n)$ and recall the notation in \eqref{fourier} for the Fourier coefficients of $V$. For any $\nu \in \N_+^n$, we define the \emph{resonant average} of $V$ with respect to the frequency vector $\nu$ as 
\begin{equation}\label{eq.res.average}
    \la V \ra_\nu(x):= \sum_{k \in \Z^2} v_{k, \nu \cdot k} e^{i k \cdot x}, \quad x \in \T^n.
\end{equation} 
\end{definition}

\begin{remark}
One can verify that this definition coincides with the earlier convention,
\begin{equation}\label{integral.average}
\la V \ra_\nu(x):= \frac{1}{2\pi} \int_{0}^{2\pi} V(t, x- \nu t) \di t.
\end{equation}
\end{remark}
The main result of this section is the following Proposition: 
\begin{proposition}\label{prop.normal.form}
There exists $\e_0>0$ sufficiently small such that, for every $0\leq \e\leq \e_0$, there exists an invertible linear map 
$$
\Phi_\e(t): \mathcal{D}^\prime(\mathbb{T}^n) \to  \mathcal{D}^\prime(\mathbb{T}^n), 
$$
and a vector field $W_{\nu,\e} \in C^\infty(\T^{n+1}; \R^n)$, with all its seminorms uniformly bounded in terms of $\e\in[0,\e_0]$ such that the following properties hold:
\begin{enumerate}
 \item $\Phi_\e(t)$ is $2\pi$-periodic,
 \item for all $\sigma\in\R_+$, there exists $C_\sigma>0$ such that, for all $0\leq\e\leq\e_0$ and for all $t \in \R$
 $$
 \left\|\Phi_\e(t)\right\|_{\cH^\sigma\to\cH^\sigma}+\left\|\Phi_\e(t)^{-1}\right\|_{\cH^\sigma\to\cH^\sigma}\leq C_\sigma,
 $$
 \item the constant $C_\sigma$ is uniformly bounded when $\sigma$ lies in a compact interval,
 \item for any solution $u_\e(t, x)$ to equation \eqref{transp.torus}, the function $v_\e(t,x)$ defined by $u_\e(t,x):=\Phi_\e(t) v_\e(t,x)$
satisfies
\begin{equation}\label{normal.form.eq}
    \pa_t v_\e(t,x)= \Opw(i \xi ( \e \la V \ra_\nu(x) + \e^2 W_{\nu,\e}(t,x))) v_\e(t,x).
\end{equation}
\end{enumerate}
\end{proposition}
Remark that equation \eqref{normal.form.eq} is a particular case of \eqref{pseudo.transport} with $b_{\e}=0$. 
Before proving this proposition, let us write the following corollary that motivates the choice of the normal form and that allows to derive unstable solutions to~\eqref{transport.torus} from the unstable solutions arising in Theorem~\ref{main.theorem} (as soon as $\langle V\rangle_\nu$ is Morse-Smale). 

\begin{corollary}\label{equivalent.growth}
Let $u_\e(t,x)$ be a solution of \eqref{transp.torus} and let $v_\e(t,x)$ be the corresponding solution of \eqref{normal.form.eq} (i.e. $u_\e=\Phi_\e(t) v_\e$ with $\Phi_\e(t)$ given in Proposition \ref{prop.normal.form}). If
\begin{equation}\label{e.growth}
    \norm{v_\e(t)}_{\sigma} \geq \tilde{\delta}_0 e^{\tilde{\delta}_0 \sigma \e t} \norm{v_\e(0)}_{\sigma}, \quad \mbox{ for some }\tilde{\delta}_0>0, \sigma\geq 0, 
\end{equation}
then one has
\begin{equation}\label{e.growth.2}
    \norm{u_\e(t)}_{\sigma} \geq C_\sigma^{-2}\tilde{\delta}_0 e^{\tilde{\delta}_0  \sigma \e t}\norm{u_\e(0)}_{\sigma}, \quad \mbox{ for some } \tilde{\delta}_0>0, \sigma\geq 0, 
\end{equation}
where $C_\sigma$ is the constant from Proposition~\ref{prop.normal.form}.
\end{corollary}
Equivalently, the behavior of the Sobolev norms of solutions to \eqref{normal.form.eq} is the same as that to \eqref{transp.torus}, up to constants. Recall that the existence of solutions to~\eqref{normal.form.eq} verifying~\eqref{e.growth} follows from Theorem~\ref{main.theorem.pseudo} as soon as $\langle V\rangle_\nu$ has the Morse-Smale property:
\begin{theorem}\label{theorem.torus.n} 
Let $\omega = 1$ and $\nu\in\N^2\setminus\{0\}$. Suppose that $\langle V\rangle_\nu$ satisifies the Morse-Smale property and, if $n=1$, suppose in addition that $\langle V\rangle_\nu$ has at least one critical point.

Then, there exists $\delta_0>0$ such that, for every $0 \leq \sigma \leq \delta_0$ and for every $0 \leq \e \leq \delta_0 \sigma^2$, one can find $v_\e \in \cH^\sigma(\T^2)$ and $r:=r(\e, \sigma, v_\e)>0$ such that 
\begin{equation}
    \norm{u_\e(0) - v_\e}_\sigma \leq r \quad \implies \quad \norm{u_\e(t)}_\sigma \geq \delta_0 e^{\delta_o \sigma t\e} \norm{u_\e (0)}_{\sigma}, \quad \mbox{ for any } t >0, 
\end{equation}
where $u_\e(t)$ is the solution to equation \eqref{transport.torus} with initial datum $u_\e(0)$.
\end{theorem}

We now turn to the proof of Proposition \ref{prop.normal.form}. 

\subsubsection{Preliminary constructions}

As already explained, we drop the index $\e$ at several places (e.g. $u,v,\Phi$ instead of $u_\e,v_\e,\Phi_\e$) in view of alleviating the notations. We now need some preliminary remarks and notation that we introduce following \cite[Section 3.1]{BambusiLangellaMontalto}. First of all, recalling the definition in \eqref{sobolev.spaces} of the Sobolev spaces $\cH^\sigma$, $\sigma \in \R$, let $A(t) : \cH^\sigma \to \cH^\sigma$ be any linear invertible transformation, periodic in time of period $2\pi$. For any differential equation of the form $\pa_t u = H u$, defining the function $v(t,x)$ as $u=A(t)v$, one can check that $v$ solves $\dot{v}=A(t)_*H v$, where the pushforward $A(t)_*H$ is given by 
\begin{equation}\label{pushforward}
A(t)_*H:=A(t)^{-1}[H A(t) -\pa_t A(t)].
\end{equation}
We remark, to avoid any confusion in the notation, that $\pa_t A(t)$ is the derivative in time of the linear transformation (and not the composition $\pa_t \circ A(t)$). 
The core of the proof of Proposition \ref{prop.normal.form} consists in choosing the proper transformation $A(t)$ and computing the corresponding pushforward in \eqref{pushforward}. To this aim, we consider a family of diffeomorphisms of the torus of the form 
\begin{equation}\label{diffeo}
\varphi(t): \T^n \to \T^n, \quad x \mapsto x +\e \beta(t, x),
\end{equation}
where $\beta \in C^{\infty}(\T \times \T^n; \R^n)$ is a function to be determined.
It has been proved (see for example \cite[Lemma B.4]{Baldi}) that, for $\e>0$ small enough, this diffeomorphism is invertible and its inverse has the form 
\begin{equation}\label{inv.diffeo}
   \varphi(t)^{-1}: \T^n \to \T^n, \quad y \mapsto y + \e \tilde{\beta}_\e(t, y),
\end{equation}
for $\tilde{\beta}_\e \in C^{\infty}(\T \times \T^n; \R^n)$. Moreover, $\tilde{\beta}_\e$ has all its derivatives uniformly bounded in terms of $\e>0$ in the admissible range of $\e$. Note that $\beta$ is chosen independently of $\e$ in our construction and that $\varphi$ depends implicitly on $\e$.

We define the transformation associated with \eqref{diffeo} as: 
\begin{equation}\label{Phi}
   \tilde{\Phi}(t)u(t, x):= \det(\text{Id} + \e \nabla_x \beta(t,x))^{\frac12} u \circ \varphi(t,x), \quad \forall u \in C^\infty(\T; \cH^\sigma(\T^n)) \quad \forall t \in \T, \forall \sigma \in \R,
\end{equation}
with inverse
\begin{equation}\label{Phi.inv}
   \tilde{\Phi}(t)^{-1}u(t, x):= \det(\text{Id} + \e \nabla_x \tilde \beta_\e(t,x))^{\frac12} u \circ \varphi^{-1}(t,x), \quad \forall u \in C^\infty(\T; \cH^\sigma(\T^n)) \quad \forall t \in \T, \forall \sigma \in \R.
\end{equation}

\noindent Remark that 
\begin{equation}\label{inv.det}
    \det(\text{Id}+ \e \nabla_y \tilde\beta_\e(t,y))= \frac{1}{\det(\text{Id} +\e\nabla_x \beta(t,x))\vert_{x=y + \e \tilde{\beta}(t, y)}}, \quad \forall y \in \T^n.
\end{equation}
Using \eqref{inv.det} and via a direct computation one can check that 
\begin{equation}\label{unitary}
    \tilde{\Phi}(t)^*=\tilde{\Phi}(t)^{-1}, \quad \forall t \in \T,
\end{equation}
where $*$ denotes the adjoint with respect to the standard inner product on $L^2$. Finally, defining 
\begin{equation}\label{H(t)}
    H(t):= (\nu + \e V(t,x)) \cdot \nabla_x + \frac{\e}{2} \div(V(t,x)), 
\end{equation}
where $\nu$ and $V$ are as in \eqref{transp.torus}, (so that equation \eqref{transp.torus} reads $\dot{u}=H(t) u$), and using formula \eqref{pushforward}, we see that the pushforward of $H(t)$ in \eqref{H(t)} with respect to $\tilde\Phi(t)$ in \eqref{Phi} is given by 
\begin{equation}\label{Phi.star}
\tilde\Phi(t)_* H = \tilde\Phi(t)^{-1} H \tilde\Phi(t) - \tilde\Phi(t)^{-1}\pa_t \tilde \Phi(t). 
\end{equation}
The following Lemma computes explicitly this pushforward. 

\begin{lemma}\label{lemma.expansion}
    Let $H$ as in \eqref{H(t)} and $\tilde{\Phi}$ as in \eqref{Phi}. Then the order one differential operator $\tilde\Phi(t)^{-1}H\tilde\Phi(t)$ is skew adjoint and the following identities hold: 
\begin{equation}\label{expand.H1}
\tilde\Phi(t)^{-1}H\tilde\Phi(t)= \Opw(i \xi( \nu + \e V(t,x) + \e \nu \cdot \nabla_x \beta +\e^2\mathcal{R}_\e^1(t,x))), 
\end{equation}
and 
\begin{equation}\label{expand.H2}
    \tilde{\Phi}(t)^{-1}\pa_t\tilde{\Phi}(t)=\Opw(i\xi(\e\pa_t \beta +\e^2\mathcal{R}_\e^2(t,x))),
\end{equation}
with $\mathcal{R}_\e^{1/2}(t,x)\in C^\infty(\T^{n+1}; \R)$  having all their seminorms uniformly bounded in terms of $\e\in[0,\e_0]$.
\end{lemma}
Here, we use the formalism of pseudodifferential operators (see section \ref{sec.pseudo}) in order to identify the symbols of $\tilde\Phi(t)^{-1}H\tilde\Phi(t)$ and $\tilde\Phi(t)^{-1}\pa_t \tilde{\Phi}(t)$, that we will need to prove Proposition \ref{prop.normal.form}. We remark that on the torus it is possible to use the Weyl quantization exactly as over $\R^n$ and that all properties in Lemma \ref{symbolic.calculus} hold. See~\cite[Ch.~5]{Zworski} for details.

\begin{proof}
First of all remark that, since $H(t)$ defined as in \eqref{H(t)} is skew adjoint, using \eqref{unitary} we have 
\[
[\tilde{\Phi}(t)^{-1}H\tilde{\Phi}(t) ]^*=\tilde\Phi(t)^* H^* (\tilde\Phi(t)^{-1})^* \stackrel{\eqref{unitary}}{=}-\tilde\Phi(t)^{-1}H\tilde{\Phi}(t),
\]
which proves skewadjointness of $\tilde{\Phi}(t)^{-1}H\tilde\Phi(t)$. Moreover, from its definition, $\tilde{\Phi}(t)^{-1}H\tilde{\Phi}(t)$ is a real valued differential operator of order one. Thus, from point 2 of Lemma \ref{symbolic.calculus}, we have
\begin{equation}\label{symbol.pt.2}
    \tilde{\Phi}(t)^{-1}H\tilde\Phi(t)=\Opw(i\xi(f_\e(t,x))),
\end{equation}
for some $f_\e \in C^\infty(\T^{n+1})$, real valued. We now compute such a function $f_\e$ in order to prove \eqref{expand.H1}. For every $u \in C^\infty(\T^{n+1})$ we have
\begin{multline*}
         \tilde\Phi(t)^{-1}H\tilde\Phi(t)u(t,x) 
        = \tilde\Phi(t)^{-1} [( \nu + \e V(t,x) ) \cdot \nabla_x (\det(\text{Id}+ \e \nabla_x \beta(t, x))^{\frac12}u(t, x+ \e \beta(t,x)))] \ \mbox{ mod } \Psi^0(\T^{n+1}) \\
         = \tilde\Phi(t)^{-1}[(\nu + \e V(t,x)) \cdot \det(\text{Id} + \e \nabla_x \beta(t,x))^{\frac12} (\text{Id}+ \e \nabla_x \beta(t,x)) \cdot \nabla_x u(t, x + \e \beta(t,x))] \ \mbox{mod} \Psi^0(\T^{n+1}) \\
         \stackrel{\substack{\eqref{Phi.inv}\\ \eqref{inv.det}}}{=}(\nu + \e V(t,x + \e \tilde{\beta}(t,x)))(\text{Id} + \e \nabla_x \beta(t, x + \e \tilde\beta(t,x))) \nabla_x u(t,x) \quad \mbox{ mod } \quad \Psi^0(\T^{n+1}).
\end{multline*}
Since we are interested in finding the first term of the expansion in $\e$ of $f_\e$ in \eqref{symbol.pt.2}, we now perform a Taylor expansion in a neighborhood of $\e=0$ and we obtain:
\begin{equation}
    \tilde\Phi(t)^{-1}H\tilde\Phi(t)u(t,x)= [(\nu + \e V(t,x))(\text{Id} + \e \nabla_x \beta(t,x)) + O_{C^\infty}(\e^2)]\nabla_x u(t,x) \quad \mbox{ mod } \quad \Psi^0(\T^{n+1}),
\end{equation} 
which gives the expansion in \eqref{expand.H1}.

We now prove \eqref{expand.H2}. We first show that there exists $g_\e \in C^\infty(\T^{n+1})$ such that 
\begin{equation}\label{skew.H2}
    \tilde\Phi(t)^{-1}\pa_t \tilde\Phi(t)= \Opw(i\xi g_\e(t,x)).
\end{equation}
This proves that  $\tilde\Phi(t)^{-1}\pa_t \tilde\Phi(t)$ is a skew adjoint operator (see Lemma \ref{symbolic.calculus}). Next, we derive the explicit expression in \eqref{expand.H2}. 
In order to prove \eqref{skew.H2}, we first remark that
\begin{equation}\label{be.careful}
\tilde\Phi(t)^{-1}\pa_t \tilde\Phi(t)[u](t,x)=\tilde\Phi(t)^{-1}\pa_t [\tilde\Phi(t)u](t,x) - \pa_t u(t,x), \quad \forall u \in C^\infty(\T^{n+1}). 
\end{equation}
Next we compute $\tilde\Phi(t)^{-1}\pa_t[\tilde\Phi(t)u](t,x)$:
\begin{equation}\label{H.3}
\begin{split}
\tilde\Phi(t)^{-1}\pa_t[\tilde\Phi(t)u](t,x) & 
\stackrel{\eqref{Phi}}{=} \tilde\Phi(t)^{-1}\pa_t[\det(\text{Id}+ \e \nabla_x \beta(t,x))^{\frac12}u(t, x +\e \beta(t,x))] \\
&= \tilde\Phi(t)^{-1}\left[\left(\pa_t \det(\text{Id}+ \e \nabla_x \beta(t,x))^{\frac12}\right)u(t,x+\e \beta(t,x))\right]\\
&+ \tilde{\Phi}(t)^{-1} \left[\det(\text{Id}+ \e \nabla_x \beta(t,x))^{\frac12}(\pa_t u(t, x +\e \beta(t,x))\right]\\
&+\tilde\Phi(t)^{-1} \left[\det(\text{Id}+ \e \nabla_x \beta(t,x))^{\frac12}\e \pa_t \beta(t,x) \cdot \nabla_x u(t, x +\e \beta(t,x)))\right]\\
& \stackrel{\eqref{Phi.inv}}{=} \left(\pa_t \det(\text{Id}+ \e \nabla_x \beta(t,x))^{\frac12}\right)\vert_{y= x + \e \tilde{\beta}_\e}\det(\text{Id}+ \e \nabla_x \tilde{\beta}_\e(t,x))^{\frac12} u(t,x) \\
&+ \pa_tu(t,x) + \e \pa_t\beta(t, x+ \e \tilde\beta_\e(t,x)) \cdot \nabla_x u(t,x).
\end{split}
\end{equation}

Thus, plugging \eqref{H.3} in \eqref{be.careful}, we obtain that
\begin{equation}\label{computation}
\begin{split}
    \tilde\Phi(t)^{-1}\pa_t \tilde\Phi(t)[u](t,x) & = \underbrace{\left(\pa_t \det(\text{Id}+ \e \nabla_x \beta(t,x))^{\frac12}\right)\vert_{y= x + \e \tilde{\beta}_\e}\det(\text{Id}+ \e \nabla_x \tilde{\beta}_\e(t,x))^{\frac12}}_{g_{\e, 1}(t,x)} u(t,x)\\
    & + \underbrace{ \e \pa_t\beta(t, x+ \e \tilde\beta_\e(t,x))}_{g_{\e, 2}(t,x)} \cdot \nabla_x u(t,x).
    \end{split}
\end{equation}

Recalling now that, from the composition rule for pseudodifferential operators (see Lemma \ref{symbolic.calculus}), for any $h(t,x) \in C^\infty(\T^{n+1})$ we have
\begin{equation}
    h(t,x) \cdot \nabla_x +\frac{\div_x h(t,x)}{2}= \Opw(i\xi \cdot h(t,x)),
\end{equation}
we are done with the proof of \eqref{skew.H2} if we show that 
\begin{equation}\label{last.to.prove}
    \frac{\div{g_{\e,2}(t,x)}}{2} = g_{\e, 1}(t,x),
\end{equation}
with $g_{\e,1}$ and ${g_{\e, 2}}$ given in \eqref{computation}. Indeed, plugging \eqref{last.to.prove} in \eqref{computation} we obtain \eqref{skew.H2} with $g_{\e}(t,x):=\e \pa_t \beta(t, x + \e \tilde{\beta}_\e(t,x))$. We show that \eqref{last.to.prove} holds with a direct computation: 
\begin{equation}\label{equat1}
\begin{split}
    g_{\e,1}(t,x) & \stackrel{\eqref{computation}}{=} \left(\pa_t \det(\text{Id}+ \e \nabla_x \beta(t,x))^{\frac12}\right)\vert_{y= x + \e \tilde{\beta}_\e}\det(\text{Id}+ \e \nabla_x \tilde{\beta}_\e(t,x))^{\frac12}\\
    & \stackrel{\eqref{inv.det}}{=}\frac12 \pa_t \det(\text{Id}+ \e \nabla_x \beta(t, x+ \e \tilde{\beta}_\e(t,x))) \det(\text{Id} + \e \nabla_x \tilde{\beta}_\e(t,x)) \\
    & = \frac12 \text{tr}\left((\text{Id}+ \e \nabla_x \beta)^{-1}\pa_t ( \text{Id} + \e \nabla_x \beta) \right) \\
    & \stackrel{\eqref{inv.det}}{=} \frac12 \text{tr}\left((\text{Id}+ \e \nabla_x \tilde{\beta}_\e) \e \pa_t  \nabla_x \beta) \right) 
\end{split}
\end{equation}
where for the unlabeled equality we used Jacobi identity. On the other hand we also have: 
 \begin{equation}\label{equat2}
 \begin{split}
     \div g_{\e,2}(t,x) & \stackrel{\eqref{computation}}{=} \div (\e \pa_t\beta(t, x+ \e \tilde\beta_\e(t,x)))\\
     & = \text{tr}\left(\e \pa_t \nabla_y \beta(t,y) \vert_{y=x +\e \tilde{\beta}_\e(t,x)} (\text{Id} + \e \nabla_x \tilde{\beta}_\e(t,x))\right).
     \end{split}
 \end{equation}
Putting together \eqref{equat1} and \eqref{equat2}, we obtain \eqref{last.to.prove}. Thus we have shown that 
\begin{equation}\label{done}
    \tilde\Phi(t)^{-1}\pa_t \tilde\Phi(t)[u](t,x)= \Opw(i \xi \cdot\e \pa_t \beta(t, x +\e \tilde{\beta}_\e(t,x))).
\end{equation}
Finally, we obtain \eqref{expand.H2} performing a Taylor expansion of \eqref{done} in a neighborhood of $\e=0$, concluding the proof of the lemma. 
\end{proof}

\subsubsection{Proof of Proposition \ref{prop.normal.form}}
Let $\tilde{\Phi}(t)$ be a transformation of the form in \eqref{Phi}, with $\beta \in C^\infty(\T^{n+1})$ to be determined and recall the expression for its inverse in \eqref{Phi.inv}. 
Take $u(t,x)$ solution of  \eqref{transp.torus}, that we rewrite as $\pa_t u = H(t) u$ with $H(t)$ in \eqref{H(t)}, and  recall from previous discussion that defining $v(t,x)$ via $u(t, x)= \tilde{\Phi}(t)v(t,x)$ then $v(t,x)$ solves 
\begin{equation}\label{first.transform}
    \pa_t v(t, x)= \tilde{\Phi}(t)_* H(t)v(t,x).
\end{equation}
See \eqref{pushforward}. We claim that it is possible to choose $\e_0>0$ small enough and $\beta \in C^\infty(\T \times \T^n; \R^n)$ such that, for every $0< \e <\e_0$, 
\begin{equation}\label{claim}
    \tilde\Phi(t)_*H(t)=\Opw(i\xi(\nu + \e V_{R_\nu}(t,x) + \e^2 \tilde{W}_{\nu,\e}(t,x))) \quad \forall t \in \T,
\end{equation}
where $\tilde{W}_{\nu,\e}(t,x) \in C^\infty(\T \times \T^n; \R^n)$ with semi-norms uniformly bounded in terms of $\e$ and 
\begin{equation}\label{V_R}
      V_{R_\nu}(t, x): = \sum_{k \in \Z^n} v_{k, \nu \cdot k }e^{\im k\cdot (x+\nu t)} \equiv \la V \ra_\nu(x+ \nu t),
\end{equation}
where $v_{k, \nu \cdot k }$ are the resonant Fourier coefficients of $V$ (see \eqref{fourier} and \eqref{eq.res.average}). 
Let us postpone the proof of this claim, and conclude the proof of the proposition assuming \eqref{claim}. Defining indeed the time translation $\cU(t): \cH^\sigma \to \cH^\sigma$, which acts as $ \cU(t)w(t,x):= w(t,x-\nu t)$, for all $t \in \R$ and for all $w \in C^\infty(\T \times \T^n; \R^n)$,
we consider 
\begin{equation}\label{u_1}
  u_1(t,x):= \cU(t)v(t,x)=v(t, x-\nu t),
\end{equation} 
where $v(t,x)$ solves equation \eqref{first.transform}. First we compute the equation solved by $u_1(t,x)$: 
\begin{equation}
\begin{split}
    \pa_t u_1(t,x) 
    & = \pa_t v(t, x-\nu t) -\nu \cdot \nabla_x v(t, x-\nu t) \\
    &= \pa_t v(t, x- \nu t) -  \Opw(i\nu\cdot\xi) v(t,x-\nu t) \\
    &  \stackrel{\substack{\eqref{first.transform}\\ \eqref{claim}}}{=} \Opw(i \xi\cdot(\e V_{R_\nu}(t, x- \nu t) + \e^2 \tilde{W}_{\nu,\e} (t, x- \nu t)))  v(t, x-\nu t)\\
    & \stackrel{\substack{\eqref{V_R}\\ \eqref{u_1}}}{=} \Opw(i\xi\cdot(\e \la V \ra_\nu (x) + \e^2  W_{\nu,\e}(t,x))) u_1(t,x),
\end{split}
\end{equation}
where $W_{\nu,\e}:= \tilde{W}_{\nu,\e}(t, x-\nu t)$. Thus $u_1(t,x)$ defined as in \eqref{u_1} solves \eqref{normal.form.eq}. Moreover the transformation $\Phi(t):= \cU(t)\tilde\Phi(t)^{-1}: \cH^\sigma \to \cH^\sigma$ is linear, $2\pi$-periodic and bounded for every $\sigma \in \R$, since both $\cU(t)$ and $\tilde\Phi(t)^{-1}$ are. Furthermore, the continuity constants are uniform in terms of $0\leq\e\leq \e_0$ (recall that $\Phi(t)$ depends implicitly on $\e$). This concludes the proof of the proposition since $u(t,x)= \Phi(t) u_1(t,x)$ by construction.

Thus we are left with finding $\e_0>0$ and $\beta \in C^\infty(\T \times \T^n; \R^n)$, so that \eqref{claim} holds. 
First of all, we plug  the expansions \eqref{expand.H1} and \eqref{expand.H2} in  \eqref{Phi.star}, obtaining 
\begin{equation}\label{2.pushforward}
        \Phi_*(t)H(t)=\Opw(i \xi\cdot( \nu + \e V(t,x)  + \e \nu \cdot \nabla_x \beta- \e \pa_t \beta + O_{C^\infty}(\e^2))),
    \end{equation}
    in a neighborhood of $\e=0$. 
    Next we consider the Fourier expansion of the term $T(t, x):=\nu +\e V +\e \nu \cdot \nabla_x \beta - \e \pa_t \beta$: 
    \begin{equation}\label{T}
        T(t, x) = \nu +\e V +\e \nu \cdot \nabla_x \beta - \e \pa_t \beta= \nu + \e \sum_{(k, \ell) \in \Z^{n+1}} ( v_{k, \ell} +\im (k \cdot \nu -\ell) \beta_{k, \ell})e^{ikx+i\ell t},
    \end{equation}
where $V(t, x)= \sum_{k, \ell} v_{k, \ell} e^{i(kx+\ell t)}$ and $\beta(t, x)= \sum_{k, \ell \in \Z} \beta_{k, \ell} e^{i(kx+\ell t)}$. Define 
\begin{equation}\label{bkl}
\beta_{k, \ell}:=
\begin{cases}
    \frac{ v_{k, \ell}}{i(\ell-k \cdot \nu)} & \mbox{ if }  \ell  \neq k \cdot \nu , \\
    0 & \mbox{ otherwise.} 
\end{cases}
\end{equation}
The corresponding $\beta$ is a well defined and smooth function over $\T^n$.
Substituting the Fourier coefficients of $\beta$ (see \eqref{bkl}), in the expression of $T(t, x)$ in \eqref{T}, we get: 
\begin{equation}\label{T.1}
T(t, x)= \nu + \e \sum_{\substack{\ell+ \nu \cdot k=0\\ k \in \Z}}v_{k, \nu \cdot k} e^{ik(x+\nu t)}=\nu  +\e V_{R_\nu}(t, x),
\end{equation}
where we recall the definition of $V_{R_\nu}$ in \eqref{V_R}. Thus recognizing that $\Phi_*(t) H(t)=\Opw(i\xi(T(t, x) +O_{C^\infty}(\e^2)))$ (se \eqref{T} and \eqref{2.pushforward}) and plugging \eqref{T.1} in \eqref{2.pushforward}, we obtain
$$
\Phi_*(t)H(t)= \Opw(i\xi(\nu+ \e V_{R_\nu}(t, x)+ \tilde{W}_{\nu,\e})),
$$
where $\tilde{W}_{\nu,\e}$ is the remainder in the order one operator.  This gives \eqref{claim}. Remark that $\e_0$ is chosen so that both \eqref{inv.diffeo} and the Taylor expansions hold for all $0< \e < \e_0$. 

\subsection{Genericity in dimension $2$ and proof of Theorem \ref{theorem.torus}}
In this last section we conclude the proof of Theorem \ref{theorem.torus}. 
Fix any $\nu \in \N_+^2$ and define 
    \begin{equation}\label{set.A.m}
        \cA_\nu:= \{ V \in C^\infty(\T^3; \R^2) : \la V \ra_\nu(x) \in \fX^{MS}(\T^2) \}.
    \end{equation}
Recall that $C^{\infty}$ is endowed with its natural Fr\'echet topology, i.e. the one induced by the seminorms $(\|.\|_{C^k})_{k\geq 0}$. One can verify that the map
    $$
    V\in C^{\infty}(\T^3;\R^2)\mapsto \langle V\rangle_\nu \in C^{\infty}(\T^2;\R^2)
    $$
is linear and continuous with respect to the Fr\'echet topology on each space. In particular, $\fX^{MS}(\T^2)$ being an open subset of $C^\infty(\T^2;\R^2)$, one finds that $\cA_\nu$ is open for the Fr\'echet topology on $C^{\infty}$. Regarding density, we let $V_0\in C^{\infty}(\T^3;\R^2)$ and we decompose it as
$$
V_0(t,x)=\langle V_0\rangle_\nu (x+t\nu)+\left(V_0(t,x)- \langle V_0\rangle_\nu (x+t\nu)\right).
$$
The second term on the right hand side has all its Fourier coefficients in $\Z^3\setminus\{(k\cdot\nu,k):k\in\Z^2\}$. In particular, it average $\langle\cdot\rangle_\nu$ is identically zero. Now recalling that $\fX^{MS}(\T^2)$ is a dense subset of $C^\infty(\T^2;\R^2)$, one can find $W$ in $\fX^{MS}(\T^2)$ arbitrarly close to $\langle V_0\rangle_\nu$. Letting 
$$
V(t,x):=W(x+t\nu)+\left(V_0(t,x)- \langle V_0\rangle_\nu (x+t\nu)\right),
$$
one finds that $V$ can be made arbitrarly close to $V_0$. This shows that $\cA_\nu$ is dense and concludes the proof of Theorem~\ref{theorem.torus} when combined with Theorem~\ref{theorem.torus.n}.

\appendix
\section{Definitions in dynamical systems}\label{Appendix.dyn}
In this appendix, we review the definitions of dynamical systems theory that we use in Sections \ref{MS.section} and \ref{section.escape}. 
 We refer to \cite{palis1998} for Morse-Smale theory and to \cite{BrinStuck, Fisher} for a complete review of the theory related to hyperbolic dynamical systems. 
Let $V$ be a smooth vector field on a smooth manifold $M$ of dimension $n \geq 1$. We denote by $\varphi_V^t(x)$ its flow at time $t \in \R$ (see \eqref{flow}).

\begin{definition}\label{hyp.crit}
A point $\Lambda \in M$ is a \emph{critical point} of $V$ if $V(\Lambda)=0$. We say that a critical point is\emph{ hyperbolic} if $\di V(\Lambda): T_\Lambda M \to T_{\Lambda}M$ has eigenvalues $\{\lambda_1, \ldots, \lambda_n\}$ such that Re$(\lambda_i) \neq 0$, for all $i=1, \ldots, n$.
\end{definition}

\begin{definition}\label{hyp.orbit}
A point $x_0\in M$ is a periodic point if $V(x_0)\neq 0$, and there exists $T_0>0$ such that $\varphi^{T_0}_V(x_0)=x_0$. Moreover, a periodic point $x_0$ is  hyperbolic if $\di\varphi_V^{T_0}(x_0)$ has $1$ as a simple eigenvalue and no other eigenvalues of modulus $1$. The set $\Lambda=\{\varphi_V^t(x_0):0\leq t\leq T_0\}$ is then called a hyperbolic closed orbit.
\end{definition}

\begin{definition}[Non-wandering set]\label{nonwand.set}
We say that a point $x \in M$ is \emph{wandering} if there exist some open neighborhood $\cO$ of $x$ and a time $T_x>0$ such that
\[
\cO \cap \left( \bigcup_{|t|>T_x}\varphi^t_V(\cO)\right)=\emptyset.
\]
We denote by $\operatorname{NW}(V)=\operatorname{NW}(\varphi_V^t)$ the \emph{non-wandering set} of $V$, i.e., the union of all points $y \in M$ which are not wandering. 
\end{definition}

\begin{definition}[$\alpha$ and $\omega$ limit sets]\label{limit.sets}
Let $x \in M$. We define
\[
\alpha(x):= \bigcap_{T \leq 0}\overline{\{\varphi_V^t(x) : t \leq T \}}, \quad \mbox{ and }
 \quad \omega(x):= \bigcap_{T \geq 0}\overline{\{\varphi_V^t(x) : t \geq T \}}.
 \]
 We remark that, if $x \in \Lambda_i$, with $ \Lambda_i \in \operatorname{Crit}(V)$ (see \eqref{critical.set}), then $\alpha(x)=\omega(x)=\Lambda_i$. 
\end{definition}

\begin{definition}[Stable and unstable manifolds]\label{s.u.man}
Let $K$ be a closed invariant subset of $M$. We define the stable and unstable manifolds associated with $K$ as 
$$
W^s(K):=\{ x \in M : \omega(x) \subset K\}\quad \mbox{ and }\quad W^u(K):= \{ x \in M : \alpha(x) \subset K \}. 
$$ 

\end{definition}

\begin{definition}\label{def.foliation}
    Let $\Lambda \in \text{Crit}(V)$ be a hyperbolic closed orbit of period $T_0$. The stable and unstable manifolds $W^{s / u}(\Lambda)$ are foliated by the family of smooth submanifolds of codimension one $\{ W^{ss/uu}(x_0)\}_{x_0 \in \Lambda}$, where \[     
    W^{ss}(x_0) := \{ x \in M \ : \lim_{n \to +\infty} \varphi^{n T_0}(x)= x_0 \},     \] 
    and 
    \[
    W^{uu}(x_0) := \{ x \in M \ : \lim_{n \to +\infty} \varphi^{-n T_0}(x)= x_0 \}.
    \]
    We remark that the same foliation can be trivially defined also in the case of critical points, for which $W^{ss/uu}(x_0)=W^{s/u}(x_0)$. 
\end{definition}

\section{Existence of solutions to transport equations}
\label{a:existence}

In this appendix, we discuss the existence of solutions to~\eqref{pseudo.transport}. Recall that the main difference with~\cite{MasperoRobert} is the presence of a non-skew adjoint term $\e \OpM(b_{\e,2})$ in the right-hand side of the equation with $b_{\e,2}(t)$ being real valued, lying in $\mathcal{C}^0_b(\R,\bS^0)$ and having all its seminorms uniformly bounded in terms of $0\leq\e\leq 1$ and $t\in\R$. To deal with this issue, we will apply a perturbative argument compared with the results from this reference. For the sake of simplicity, we set
$$
\mathcal{B}_0(t)=\OpM(i\xi(V(x)+\e\mathcal{P}(t,x))+ib_{\e,1}),\ \text{and}\ 
\mathcal{R}_0(t)=\e\OpM( b_{\e,2}).
$$
We denote by $\mathcal{U}_0(t,s)\in\mathcal{C}^0(\R\times\R;\mathcal{L}(\cH^{\sigma}))$ the flow associated with $\mathcal{B}_0(t)$. Recall from~\cite[Th.~1.2]{MasperoRobert} that it satisfies 
\begin{equation}\label{e.masperorobert}
\|\mathcal{U}_0(t,s)\|_{\cH^\sigma\rightarrow \cH^\sigma}\leq C_\sigma e^{C_\sigma|t-s|}
\end{equation} for every $\sigma\geq 0$ and $t,s\in\R$ and that it is unitary in $L^2$ (i.e., for $\sigma=0$). Thanks to Duhamel's principle, the equation we are interested in can be rewritten under an integral form, for every $t$ in $\R$,
$$
u_\e(t)=\mathcal{U}_0(t,0)u_\e(0)+\int_0^t\mathcal{U}_0(t,\tau)\mathcal{R}_0(\tau)u_\e(\tau)d\tau.
$$
Hence, we fix $u_0\in\cH^\sigma$, $T>0$ and we consider the functional 
$$
F:u\in\mathcal{C}^0([-T,T];\cH^\sigma)\mapsto \mathcal{U}_0(t,0)u_0+\int_0^t\mathcal{U}_0(t,\tau)\mathcal{R}_0(\tau)u(\tau)d\tau \in\mathcal{C}^0([-T,T];\cH^\sigma).
$$
Our goal is to prove the existence of a fixed point $F(v)=v$ in $\cH^\sigma$. To do that, we proceed by induction and we set 
$$v_0=u_0,\quad v_{n+1}=F(v_n).$$   
One has
$$
\|v_{n+1}(t)-v_n(t)\|_\sigma\leq C_{\sigma,T} \int_0^t\|v_{n}(\tau)-v_{n-1}(\tau)\|_\sigma d\tau, \quad \forall -T \leq t \leq T, \ \forall n \in \N,
$$
where we used~\eqref{e.masperorobert} and the property that $\mathcal{R}_0(\tau):\cH^\sigma\rightarrow \cH^{\sigma}$ is a bounded operator thanks to~\eqref{e.cv}. We remark that $C_{\sigma,T}>0$ depends only on $\sigma$, $T$ and the norms of $\mathcal{B}_0$ and $\mathcal{R}_0$. By induction, one finds
$$
\|v_{n+1}(t)-v_n(t)\|_\sigma\leq \frac{(C_{\sigma,T}t)^n}{n!} \sup_{\tau\in[-T,T]}\|v_{1}(\tau)-v_{0}(\tau)\|_\sigma ,
$$
from which we can infer that $(v_n)_{n\geq 0}$ is a Cauchy sequence in $\mathcal{C}^0([-T,T],\cH^\sigma)$. By continuity of the functional $F$, we find that the limit $v_\infty$ satisfies $F(v_\infty)=v_\infty$. The same argument shows that the solution is unique. As $T$ can be chosen arbitrarily, we find that there exists an unique solution to~\eqref{pseudo.transport} that is defined on $\R$.

\small

\bibliographystyle{alpha}

\Addresses
\end{document}